\documentclass[a4]{amsart}

\usepackage{amssymb}
\usepackage[all]{xy}
\usepackage{amsmath, amsfonts, amsthm , amssymb} 

\input xypic

\newtheorem{theorem}{Theorem}[section]
\newtheorem{proposition}[theorem]{Proposition}
\newtheorem{lemma}[theorem]{Lemma}
\newtheorem{corollary}[theorem]{Corollary}

\theoremstyle{definition}
\newtheorem{definition}[theorem]{Definition}
\newtheorem{example}[theorem]{Example} 
\newtheorem{examples}[theorem]{Examples}
\newtheorem{remark}[theorem]{Remark}

\newcommand{\uxa}{\ensuremath{(\underline{X},\underline{A})}} 
\newcommand{\cxx}{\ensuremath{(\underline{CX},\underline{X})}} 
\newcommand{\cxxk}{\ensuremath{(\underline{CX},\underline{X})^{K}}} 
\newcommand{\rest}{\ensuremath{\mbox{rest}}} 
\newcommand{\starr}{\ensuremath{\mbox{star}}} 
\newcommand{\link}{\ensuremath{\mbox{link}}} 
\newcommand{\W}{\ensuremath{\mathcal{W}}}

\newcounter{bean}
\newenvironment{letterlist}{\begin{list}{\rm ({\alph{bean}})}
      {\usecounter{bean}\setlength{\rightmargin}{\leftmargin}}}
      {\end{list}}

\newcommand{\namedright}[3]{\ensuremath{#1\stackrel{#2}
 {\longrightarrow}#3}}
\newcommand{\nameddright}[5]{\ensuremath{#1\stackrel{#2}
 {\longrightarrow}#3\stackrel{#4}{\longrightarrow}#5}}
\newcommand{\namedddright}[7]{\ensuremath{#1\stackrel{#2}
 {\longrightarrow}#3\stackrel{#4}{\longrightarrow}#5
  \stackrel{#6}{\longrightarrow}#7}}

\newcommand{\larrow}{\relbar\!\!\relbar\!\!\rightarrow}
\newcommand{\llarrow}{\relbar\!\!\relbar\!\!\larrow}

\newcommand{\lnamedright}[3]{\ensuremath{#1\stackrel{#2}
 {\larrow}#3}}

\newcommand{\llnamedddright}[7]{\ensuremath{#1\stackrel{#2}
 {\llarrow}#3\stackrel{#4}{\llarrow}#5
  \stackrel{#6}{\llarrow}#7}}

\newcommand{\qqed}{\hfill\Box}

\begin{document}
\title{The homotopy type of the polyhedral product for shifted complexes}
\author{Jelena Grbi\'{c}}
\address{School of Mathematics, University of Manchester,
         Manchester M13 9PL, United Kingdom}
\email{Jelena.Grbic@manchester.ac.uk}
\author{Stephen Theriault}
\address{Institute of Mathematics,
         University of Aberdeen, Aberdeen AB24 3UE, United Kingdom}
\email{s.theriault@abdn.ac.uk}

\subjclass[2000]{Primary 13F55, 55P15, Secondary 52C35.}
\date{}
\keywords{Davis-Januszkiewicz space, moment-angle complex,
   polyhedral product, homotopy type}

%%% Abstract
\begin{abstract}
We prove a conjecture of Bahri, Bendersky, Cohen and Gitler:
if $K$ is a shifted simplicial complex on $n$ vertices, $X_{1},\cdots, X_{n}$ are
spaces and $CX_{i}$ is the cone on $X_{i}$, then the polyhedral product
determined by $K$ and the pairs $(CX_{i},X_{i})$ is homotopy equivalent
to a wedge of suspensions of smashes of the $X_{i}$'s. This generalises 
earlier work of the two authors in the special case where each $X_{i}$ is a loop space. Connections are made to toric topology, combinatorics, and classical 
homotopy theory. 
\end{abstract} 

\maketitle

\section{Introduction} 

Polyhedral products generalize the notion of a product of spaces.  
They are of widespread interest due to their being fundamental objects 
which arise in many areas of mathematics. For example, in certain 
dynamical systems they arise as invariants of the system, in robotics 
they are related to configuration spaces of planar linkages, in 
combinatorics they appear as the complements of complex coordinate 
subspace arrangements, and in algebraic geometry they appear as 
certain intersections of quadrics. Their topological properties have 
attracted a great deal of recent attention due in part to their emergence 
as central objects of study in toric topology. This includes work  
on their geometric properties~\cite{BP1,BP2}, homology~\cite{BP1,DS},  
their rational homotopy~\cite{FT,NR}, and their homotopy 
types~\cite{BBCG1,BBCG2,GT1,GT2}. 

To define a polyhedral product, let $K$ be a simplicial complex on the index 
set $[n]$. For $1\leq i\leq n$, let $(X_{i},A_{i})$ be a pair of pointed 
$CW$-complexes, where $A_{i}$ is a pointed subspace of $X_{i}$. Let 
$\uxa=\{(X_{i},A_{i})\}_{i=1}^{n}$ be the sequence of pairs. For each 
simplex $\sigma\in K$, let $\uxa^{\sigma}$ be the subspace of 
$\prod_{i=1}^{n} X_{i}$ defined by 
\[\uxa^{\sigma}=\prod_{i=1}^{n} Y_{i}\qquad\qquad 
       \mbox{where}\ Y_{i}=\left\{\begin{array}{ll} 
                                             X_{i} & \mbox{if $i\in\sigma$} \\ 
                                             A_{i} & \mbox{if $i\notin\sigma$}. 
                                       \end{array}\right.\] 
The \emph{polyhedral product} determined by \uxa\ and $K$ is 
\[\uxa^{K}=\bigcup_{\sigma\in K}\uxa^{\sigma}\subseteq\prod_{i=1}^{n} X_{i}.\] 
For example, suppose each $A_{i}$ is a point. If $K$ is a disjoint union 
of $n$ points then $(\underline{X},\underline{\ast})^{K}$ is the wedge 
$X_{1}\vee\cdots\vee X_{n}$, and if $K$ is the standard $(n-1)$-simplex 
then $(\underline{X},\underline{\ast})^{K}$ is the product 
$X_{1}\times\cdots\times X_{n}$. 

The case $(\underline{X},\underline{\ast})^{K}$ is related to another case 
of particular interest. Observe that any polyhedral product $\uxa^{K}$ 
is a subset of the product $X_{1}\times\cdots\times X_{n}$. In the 
special case $(\underline{X},\underline{\ast})^{K}$, Denham and Suciu~\cite{DS} 
show that there is a homotopy fibration (that is, a fibration, up to homotopy) 
\begin{equation} 
  \label{cxxfib} 
  \nameddright{(\underline{C\Omega X},\underline{\Omega X})^{K}}{} 
    {(\underline{X},\underline{\ast})^{K}}{}{\prod_{i=1}^{n} X_{i}} 
\end{equation}  
where $C\Omega X$ is the cone on $\Omega X$. Special cases of 
this fibration recover some classical results in homotopy theory. 
For example, if $K$ is two distinct points, then 
$(\underline{C\Omega X},\underline{\Omega X})^{K}$ is the 
fibre of the inclusion 
\(\namedright{X_{1}\vee X_{2}}{}{X_{1}\times X_{2}}\). 
Ganea~\cite{G} identified the homotopy type of this fibre as 
$\Sigma\Omega X_{1}\wedge\Omega X_{2}$. If $K=\Delta^{n-1}_{k}$ 
is the full $k$-skeleton of the standard $n$-simplex, then Porter~\cite{P1} 
showed that for $0\leq k\leq n-2$ there is a homotopy equivalence 
\[(\underline{C\Omega X},\underline{\Omega X})^{K}\simeq\bigvee_{j=k+2}^{n}
      \left(\bigvee_{1\leq i_{1}<\cdots<i_{j}\leq n}\ \binom{j-1}{k+1}
      \Sigma^{k+1}\Omega X_{i_{1}}\wedge\cdots\wedge\Omega X_{i_{j}}\right),\] 
where $j\cdot Y$ denotes the wedge sum of $j$ copies of the space $Y$. 

The emergence of toric topology in the late 1990's brought renewed attention 
to these classical results, in a new context. Davis and Januszkiewicz~\cite{DJ} 
constructed a new family of manifolds with a torus action. The construction 
started with a simple convex polytope $P$ on $n$ vertices, passed to the 
simplicial complex $K=\partial P^{\ast}$ - the boundary of the dual of $P$,  
and associated to it a manifold~$\mathcal{Z}_{K}$ with a torus action and 
an intermediate space $DJ(K)$, which fit into a homotopy fibration 
\[\nameddright{\mathcal{Z}_{K}}{}{DJ(K)}{}{\prod_{i=1}^{n}\mathbb{C}P^{\infty}}.\] 
Buchstaber and Panov~\cite{BP1} recognized the space $DJ(K)$ as 
$(\underline{\mathbb{C}P}^{\infty},\underline{\ast})^{K}$, and this allowed 
them to generalize Davis and Januszkiewicz's construction to a homotopy fibration  
\[\nameddright{\mathcal{Z}_{K}}{}{DJ(K)}{}{\prod_{i=1}^{n}\mathbb{C}P^{\infty}}\] 
for any simplicial complex $K$ on $n$ vertices. Here, 
$DJ(K)=(\underline{\mathbb{C}P}^{\infty},\underline{\ast})^{K}$ and 
$\mathcal{Z}_{K}=(\underline{D}^{2},\underline{S}^{1})^{K}$. The 
spaces $DJ(K)$ and $\mathcal{Z}_{K}$ are central objects of study in toric topology, 
and their thorough study in~\cite{BP1,BP2} launched toric topology into the mainstream 
of modern algebraic topology. 
The generalization to polyhedral products soon followed in unpublished notes 
by Strickland and under the name $K$-powers in~\cite{BP1}, and came 
to prominence in recent work of Bahri, Bendersky, Cohen and Gitler~\cite{BBCG1}. 

Following Ganea's and Porter's results, it is natural to ask when the 
homotopy type of the fibre $(\underline{C\Omega X},\underline{\Omega X})^{K}$ 
in~(\ref{cxxfib}) can be recognized. It is too ambitious to hope to do this 
for all $K$, but it is reasonable to expect that it can be done for certain 
families of simplicial complexes. This is precisely what was done in 
earlier work of the authors. A simplicial complex $K$ is 
\emph{shifted} if there is an ordering on its vertices such that whenever 
$\sigma\in K$ and $\nu^{\prime}<\nu$, then 
$(\sigma-\nu)\cup\nu^{\prime}\in K$. This is a fairly large family 
of complexes, which includes Porter's case of full $k$-skeletons 
of a standard $n$-simplex. In~\cite{GT1} it was shown that if $K$ 
is shifted, then there is a homotopy equivalence 
\begin{equation} 
  \label{GTequiv} 
  (\underline{C\Omega X},\underline{\Omega X})^{K}\simeq 
        \bigvee_{\alpha\in\mathcal{I}}\Sigma^{\alpha(t)}\Omega X_{1}^{(\alpha_{1})} 
        \wedge\cdots\wedge\Omega X_{n}^{(\alpha_{n})} 
\end{equation}  
for some index set $\mathcal{I}$ (which can be made explicit),  where 
if $\alpha_{i}=0$ then the smash product is interpreted as omitting the 
factor $X_{i}$ rather than being trivial. 
The homotopy equivalence~(\ref{GTequiv}) has implications in combinatorics. 
In~\cite{BP1}, it was shown that $\mathcal{Z}_{K}$ is homotopy equivalent to 
the complement of the coordinate subspace arrangement determined by $K$. 
Such spaces have a long history of study by combinatorists. In particular, 
as $\mathcal{Z}_{K}=(\underline{D}^{2},\underline{S}^{1})^{K}$, the homotopy 
equivalence~(\ref{GTequiv}) implies that $\mathcal{Z}_{K}$ is homotopy 
equivalent to a wedge of spheres, which answered a major outstanding 
problem in combinatorics. 

Bahri, Bendersky, Cohen and Gitler~\cite{BBCG1} gave a general decomposition 
of $\Sigma\uxa^{K}$, which in the special case of $\cxx^{K}$ is as follows. 
Regard the simplices of $K$ as ordered sequences, $(i_{1},\ldots,i_{k})$ where 
$1\leq i_{1}<\cdots<i_{k}\leq n$. Let 
$\widehat{X}^{I}=X_{i_{1}}\wedge\cdots\wedge X_{i_{k}}$. Let $Y\ast Z$ be 
the \emph{join} of the topological spaces $X$ and $Y$, and recall that 
there is a homotopy equivalence $Y\ast Z\simeq\Sigma Y\wedge Z$. Let 
$K_{I}\subseteq K$ be the full subcomplex of $K$ consisting of the 
simplices in $K$ which have all their vertices in $K$, that is, 
$K=\{\sigma\cap I\mid\sigma\in K\}$. Let $\vert K\vert$ be the 
geometric realization of the simplicial complex~$K$. Then for any 
simplicial complex $K$, there is a homotopy equivalence 
\begin{equation} 
   \label{BBCGdecomp} 
   \Sigma\cxx^{K}\simeq\Sigma\left( 
      \bigvee_{I\notin K}\vert K_{I}\vert\ast\widehat{X}^{I}\right). 
\end{equation}  
In particular, (\ref{BBCGdecomp})~agrees with the suspension of 
the homotopy equivalence in~(\ref{GTequiv}) in the case of 
$(\underline{C\Omega X},\underline{\Omega X})^{K}$. Bahri, 
Bendersky, Cohen and Gitler conjectured that if $K$ is shifted 
then~(\ref{BBCGdecomp}) desuspends. Our main result is that 
this conjecture is true. 

\begin{theorem} 
   \label{cxxshifted} 
   Let $K$ be a shifted-complex. Then there is a homotopy equivalence 
   \[\cxxk\simeq\left(\bigvee_{I\notin K}\vert K_{I}\vert\ast\widehat{X}^{I}\right).\] 
\end{theorem} 

The methods used to prove the results in~\cite{GT1} in the case 
$(\underline{C\Omega X},\underline{\Omega X})^{K}$ involved 
analyzing properties of the fibration~(\ref{cxxfib}). In the general 
case of $\cxx^{K}$, no such fibration exists, so we need to develop 
new methods. An added benefit is that these new methods also 
give a much faster proof of the results in~\cite{GT1}. As well, 
in Sections~\ref{sec:gluing} and~\ref{sec:wedge} we extend our 
methods to desuspend~(\ref{BBCGdecomp}) in cases where $K$ 
is not shifted.

\section{A special case} 
\label{sec:join} 

Let $\Delta^{n-1}$ be the standard $n$-simplex. For $0\leq k\leq n-1$, 
let $\Delta^{n-1}_{k}$ be the full $k$-skeleton of $\Delta^{n-1}$. In 
this brief section we will identify \cxxk\ when $K=\Delta^{n-1}_{n-2}$. 
We begin with some general observations which hold for any \uxa.   

\begin{lemma} 
   \label{inclusion} 
   Let $K$ be a simplicial complex on $n$ vertices. Let 
   $\sigma_{1},\sigma_{2}\in K$ and suppose that 
   $\sigma_{1}\subseteq\sigma_{2}$. 
   Then $\uxa^{\sigma_{1}}\subseteq\uxa^{\sigma_{2}}$. 
\end{lemma} 

\begin{proof} 
By definition, $\uxa^{\sigma_{1}}=\prod_{i=1}^{n} Y_{i}$ 
where $Y_{i}=X_{i}$ if $i\in\sigma_{1}$ and $Y_{i}=A_{i}$ if $i\notin\sigma_{1}$. 
Similarly, $\uxa^{\sigma_{2}}=\prod_{i=1}^{n} Y^{\prime}_{i}$ 
where $Y^{\prime}_{i}=X_{i}$ if $i\in\sigma_{2}$ and $Y^{\prime}_{i}=A_{i}$ 
if $i\notin\sigma_{2}$. Since $\sigma_{1}\subseteq\sigma_{2}$, 
if $i\in\sigma_{1}$ then $i\in\sigma_{2}$ so $Y_{i}=Y^{\prime}_{i}$. 
On the other hand, if $i\notin\sigma_{1}$ then $Y_{i}=A_{i}$, 
implying that $Y_{i}\subseteq Y^{\prime}_{i}$. Thus 
$\prod_{i=1}^{n} Y_{i}\subseteq\prod_{i=1}^{n} Y^{\prime}_{i}$ and 
the lemma follows. 
\end{proof} 

A face $\sigma\in K$ is called \emph{maximal} if there is no 
other face $\sigma^{\prime}\in K$ with the property that 
$\sigma\subsetneq\sigma^{\prime}$. In other words, a  
non-maximal face of $K$ is a proper subset of another face 
of $K$. So $K$ is the union of its maximal faces. Lemma~\ref{inclusion} 
then immediately implies the following. 

\begin{corollary} 
   \label{maxunion} 
   There is an equality of sets 
   $\uxa^{K}=\bigcup_{\sigma\in\mathcal{I}}\uxa^{\sigma}$ 
   where $\mathcal{I}$ runs over the list of maximal faces of~$K$.~$\qqed$  
\end{corollary} 
 
For example, let $K=\Delta^{n-1}_{n-2}$. The maximal faces of $K$ 
are $\bar{\sigma}_{i}=(1,\ldots,\hat{i},\ldots,n)$ for $1\leq i\leq n$, 
where $\hat{i}$ means omit the $i^{th}$-coordinate. Thus 
$K=\bigcup_{i=1}^{n}\bar{\sigma}_{i}$ and Corollary~\ref{maxunion} 
implies that $\uxa^{K}=\bigcup_{i=1}^{n}\underline{X}^{\bar{\sigma}_{i}}$. 
Explicitly, we have  
$\underline{X}^{\bar{\sigma}_{i}}=X_{1}\times\cdots\times A_{i}\times\cdots\times X_{n}$ 
so 
\[\uxa^{K}=\bigcup_{i=1}^{n} X_{1}\times\cdots\times A_{i}\times\cdots\times X_{n}.\]  

As a special case, consider \cxxk. Then 
\begin{equation} 
   \label{cxxfatwedge} 
   \cxxk=\bigcup_{i=1}^{n}\underline{CX}^{\bar{\sigma}_{i}}= 
         \bigcup_{i=1}^{n} CX_{1}\times\cdots\times X_{i}\times\cdots\times CX_{n}. 
\end{equation} 
Porter~\cite[Appendix, Theorem 3]{P1} showed that there is a homotopy 
equivalence 
\[\Sigma^{n-1} X_{1}\wedge\cdots\wedge X_{n}\simeq 
       \bigcup_{i=1}^{n} CX_{1}\times\cdots\times X_{i}\times\cdots\times CX_{n}.\] 
Thus we obtain the following. 

\begin{proposition} 
   \label{join} 
   Let $K=\Delta^{n-1}_{n-2}$. Then there is a homotopy equivalence 
   \[\cxxk\simeq\Sigma^{n-1} X_{1}\wedge\cdots\wedge X_{n}.\] 
   $\qqed$ 
\end{proposition}

\section{Some general properties of polyhedral products} 
\label{sec:ppproperties} 

In this section we establish some general properties of polyhedral 
products which will be used later. First, we consider 
how the polyhedral product functor behaves with respect to a union 
of simplicial complexes. Let $K$ be a simplicial complex on $n$ vertices 
and suppose that $K=K_{1}\cup_{L} K_{2}$. 
Relabelling the vertices if necessary, we may assume that $K_{1}$ 
is defined on the vertices $\{1,\ldots,m\}$, $K_{2}$ is defined on 
the vertices $\{m-l+1,\ldots,n\}$ and $L$ is defined on the 
vertices $\{m-l+1,\ldots,m\}$. By including the vertex set 
$\{1,\ldots,m\}$ into the vertex set $\{1,\ldots,n\}$, we may 
regard $K_{1}$ as a simplicial complex on $n$ vertices. Call the 
resulting simplicial complex on $n$ vertices $\overline{K}_{1}$. 
Similarly, we may define simplicial complexes $\overline{K}_{2}$ 
and $\overline{L}$ on $n$ vertices. Then we have 
$K=\overline{K}_{1}\cup_{\overline{L}}\overline{K}_{2}$. The point  
in doing this is that the four objects $K$, $\overline{K}_{1}$, 
$\overline{K}_{2}$ and $\overline{L}$ are now in the same category 
of simplicial complexes on $n$ vertices, so we may apply the 
polyhedral product functor.   

\begin{proposition} 
   \label{uxaunion} 
   Let $K$ be a finite simplical complex on $n$ vertices. Suppose 
   $K=K_{1}\cup_{L} K_{2}$ where  $L=K_{1}\cap K_{2}$. Then 
   \[\uxa^{K}=\uxa^{\overline{K}_{1}}\cup_{\uxa^{\overline{L}}}\uxa^{\overline{K}_{2}}.\] 
\end{proposition} 

\begin{proof} 
Since $K=K_{1}\cup_{L} K_{2}$ and $K$ is finite, the faces in $K$ can be 
put into three finite collections: (A) the faces in $L$, (B) the faces in $K_{1}$ 
that are not faces of $L$ and (C) the faces of $K_{2}$ that are not faces 
of $L$. Thus we have 
\begin{eqnarray*} 
      L & = & \bigcup_{\sigma\in A}\sigma \\  
      K_{1} & = & \left(\bigcup_{\sigma\in A}\right)\cup
             \left(\bigcup_{\sigma^{\prime}\in B}\sigma^{\prime}\right) \\ 
      K_{2} & = & \left(\bigcup_{\sigma\in A}\right)\cup 
             \left(\bigcup_{\sigma^{\prime\prime}\in C}\sigma^{\prime\prime}\right) \\  
      K & = & \left(\bigcup_{\sigma\in A}\sigma\right)\cup 
            \left(\bigcup_{\sigma^{\prime}\in B}\sigma^{\prime}\right) 
            \cup\left(\bigcup_{\sigma^{\prime\prime}\in C}\sigma^{\prime\prime}\right). 
\end{eqnarray*} 
By definition, for any simplicial complex $M$ on $n$ vertices, 
$\uxa^{M}=\bigcup_{\sigma\in M}\uxa^{\sigma}$. So in our case, we have  
\begin{eqnarray*} 
    \uxa^{\overline{L}} & = & \bigcup_{\sigma\in A}\uxa^{\sigma} \\  
    \uxa^{\overline{K}_{1}} & = & \left(\bigcup_{\sigma\in A}\uxa^{\sigma}\right)\cup 
         \left(\bigcup_{\sigma^{\prime}\in B}\uxa^{\sigma^{\prime}}\right) \\  
    \uxa^{\overline{K}_{2}} & = & \left(\bigcup_{\sigma\in A}\uxa^{\sigma}\right)\cup 
         \left(\bigcup_{\sigma^{\prime\prime}\in C}\uxa^{\sigma^{\prime\prime}}\right) \\  
    \uxa^{K} & = & \left(\bigcup_{\sigma\in A}\uxa^{\sigma}\right)\cup 
        \left(\bigcup_{\sigma^{\prime}\in B}\uxa^{\sigma^{\prime}}\right) 
       \cup\left(\bigcup_{\sigma^{\prime\prime}\in C}\uxa^{\sigma^{\prime\prime}}\right). 
\end{eqnarray*}  
In particular,  
$\uxa^{K}=\uxa^{\overline{K}_{1}}\cup\uxa^{\overline{K}_{2}}$ 
and 
$\uxa^{\overline{L}}=\uxa^{\overline{K}_{1}}\cap\uxa^{\overline{K}_{2}}$. 
That is, 
\[\uxa^{K}=\uxa^{\overline{K}_{1}}\cup_{\uxa^{\overline{L}}}\uxa^{\overline{K}_{2}}.\] 
\end{proof} 

It is appealing to slightly alter the statement of Proposition~\ref{uxaunion}. 

\begin{corollary} 
   \label{uxaunioncor} 
   If $K=K_{1}\cup_{L} K_{2}$ is a simplicial complex on $n$ vertices then 
   \[\uxa^{\overline{K}_{1}\cup_{\overline{L}}\overline{K}_{2}}= 
             \uxa^{\overline{K}_{1}}\cup_{\uxa^{\overline{L}}}\uxa^{\overline{K}_{2}}.\] 
   That is, the polyhedral product functor commutes with pushouts.~$\qqed$  
\end{corollary} 

Next, suppose $K$ is a simplicial complex on the index set $[n]$. 
Let $L$ be a subcomplex of $K$. Reordering the indices of necessary, 
assume that the vertices of $L$ are $\{1,\ldots,m\}$ for $m\leq n$. 
For the application we have in mind, specialize to $\cxx^{K}$. 
Let $\widehat{X}=\prod_{i=m+1}^{n} X_{i}$. Since the indices 
of the factors in $\widehat{X}$ are complementary to the vertex 
set $\{1,\ldots,m\}$ of $L$, the inclusion 
\(\namedright{L}{}{K}\) 
induces an inclusion 
\(I\colon\namedright{\cxx^{L}\times\widehat{X}}{}{\cxx^{K}}\). 
In Proposition~\ref{projincl} we show that the restriction of~$I$ 
to~$\widehat{X}$ is null homotopic. We first need a preparatory lemma.  

\begin{lemma} 
   \label{projincllemma} 
   The inclusion 
   \[J\colon\namedright{X_{1}\times\cdots\times X_{n}}{} 
        {\bigcup_{i=1}^{n} X_{1}\times\cdots\times CX_{i}\times\cdots\times X_{n}}\] 
   is null homotopic. 
\end{lemma}  

\begin{proof} 
For $1\leq k\leq n$, let 
$F_{k}=\bigcup_{i=1}^{k} X_{1}\times\cdots\times CX_{i}\times\cdots\times X_{n}$. 
Then $F_{1}\subseteq F_{2}\subseteq\cdots\subseteq F_{n}$, and  
$\{F_{k}\}_{k=1}^{n}$ is a filtration of 
$\bigcup_{i=1}^{n} X_{1}\times\cdots\times CX_{i}\times\cdots\times X_{n}$. 
Observe that $J$ factors as a composite of inclusions 
\(\nameddright{X_{1}\times\cdots\times X_{n}}{}{F_{1}}{}{F_{2}}\longrightarrow 
       \cdots\longrightarrow F_{n}\). 

Consider first the inclusion 
\(\namedright{X_{1}\times\cdots\times X_{n}}{} 
      {F_{1}=CX_{1}\times X_{2}\times\cdots\times X_{n}}\). 
The cone in the first coordinate of $F_{1}$ implies that this inclusion  
is homotopic to the composite 
\(\nameddright{X_{1}\times\cdots\times X_{n}}{\pi_{1}} 
      {X_{2}\times\cdots\times X_{n}}{\varphi_{1}} 
      {CX_{1}\times X_{2}\times\cdots\times X_{n}}\), 
where $\pi_{1}$ is the projection and $\varphi_{1}$ is the inclusion. 
Composing into 
$F_{2}=CX_{1}\times X_{2}\times\cdots\times X_{n}\cup 
      X_{1}\times CX_{2}\times X_{3}\times\cdots\times X_{n}$, 
we obtain a homotopy commutative diagram 
\[\diagram 
        & X_{2}\times\cdots\times X_{n}\rto\dto^{\varphi_{1}} 
            & CX_{2}\times X_{3}\times\cdots\times X_{n}\dto \\ 
        X_{1}\times\cdots\times X_{n}\rto\urto^{\pi_{1}} 
             & F_{1}\rto & F_{2}  
  \enddiagram\] 
where the square strictly commutes and each map in the square 
is an inclusion. As before, the map 
\(\namedright{X_{2}\times\cdots\times X_{n}}{} 
     {CX_{2}\times X_{3}\times\cdots\times X_{n}}\) 
in the top row to is homotopic to the composite 
\(\nameddright{X_{2}\times\cdots\times X_{n}}{} 
      {X_{3}\times\cdots\times X_{n}}{}{CX_{2}\times X_{2}\times\cdots\times X_{n}}\) 
where the left map is the projection and the right map is the 
inclusion. Thus the inclusion 
\(\namedright{X_{1}\times\cdots\times X_{n}}{}{F_{2}}\) 
is homotopic to the composite 
\(\nameddright{X_{1}\times\cdots\times X_{n}}{\pi_{2}} 
       {X_{3}\times\cdots\times X_{n}}{\varphi_{2}}{F_{2}}\), 
where $\pi_{2}$ is the projection and $\varphi_{2}$ is an 
inclusion. Iterating, we obtain that the inclusion 
\(\namedright{X_{1}\times\cdots\times X_{n}}{j}{F_{n}}\) 
is homotopic to the composite 
\(\nameddright{X_{1}\times\cdots\times X_{n}}{\pi_{n}}{\ast}{\varphi_{n}}{F_{n}}\) 
where $\pi_{n}$ is the projection and $\varphi_{n}$ is the inclusion. 
Hence $J$ is null homotopic. 
\end{proof} 

\begin{proposition} 
   \label{projincl} 
   Let $K$ be a simplicial complex on the index set $[n]$ and 
   let $L$ be a subcomplex of~$K$ on $[m]$, where $m\leq n$. 
   Suppose that each vertex $\{i\}\in K$ for $m+1\leq i\leq n$. 
   Let $\widehat{X}=\prod_{i=m+1}^{n} X_{i}$. Then the restriction of 
   \(\namedright{\cxx^{L}\times\widehat{X}}{I}{\cxx^{K}}\) 
   to $\widehat{X}$ is null homotopic.  
\end{proposition} 

\begin{proof}  
By definition of the polyhedral product, 
$\cxx^{\{i\}}=X_{1}\times\cdots\times CX_{i}\times\cdots X_{n}$. 
Since each vertex $\{i\}\in K$ for $m+1\leq i\leq n$, we obtain an inclusion 
\[I^{\prime\prime}\colon\namedright  
      {\bigcup_{i=m+1}^{n} X_{1}\times\cdots\times CX_{i}\times\cdots X_{n}}{}{\cxxk}.\] 
As the indices $\{m+1,\ldots,n\}$ are complementary to the vertex set 
$\{1,\ldots,m\}$ of $L$, the restriction of $I^{\prime\prime}$ to 
$X_{1}\times\cdots\times X_{m}$ factors through 
\(i\colon\namedright{\cxx^{L}}{}{\cxx^{K}}\). 
Thus we can take the coordinate-wise product of $i$ and $I^{\prime\prime}$ 
to obtain an inclusion 
\[I^{\prime}\colon\namedright{\cxx^{L}\times\left(  
    \bigcup_{i=1}^{n} X_{m+1}\times\cdots\times CX_{i}\times\cdots X_{n}\right)}{}{\cxxk}.\]  
Observe that $I=I^{\prime}\circ J^{\prime}$, where $J^{\prime}$ is the inclusion  
\[J^{\prime}\colon\lnamedright{\cxx^{L}\times X_{m+1}\times\cdots\times X_{n}}{1\times J} 
        {\cxx^{L}\times\left(\bigcup_{i=m+1}^{n} 
        X_{m+1}\times\cdots\times CX_{i}\times\cdots\times X_{n}}\right).\] 
By Lemma~\ref{projincllemma}, $J$ is null homotopic. Thus 
the restriction of $J^{\prime}$ to $\widehat{X}=X_{m+1}\times\cdots\times X_{n}$ 
is null homotopic. Therefore the restriction of $I$ to $\widehat{X}$ is null 
homotopic. 
\end{proof} 

It is worth pointing out the special case when $L=\emptyset$. By the definition 
of the polyhedral product, $\cxx^{\emptyset}=X_{1}\times\cdots\times X_{n}$. 
Considering the inclusion 
\(\namedright{\cxx^{\emptyset}}{}{\cxx^{K}}\), 
Proposition~\ref{projincl} immediately implies the following. 

\begin{corollary} 
   \label{projinclcor} 
   Let $K$ be a simplicial complex on the index set $[n]$ and suppose that 
   each vertex is in $K$. Then the inclusion 
   \(\namedright{X_{1}\times\cdots\times X_{n}}{}{\cxx^{K}}\) 
   is null homotopic.~$\qqed$ 
\end{corollary}

\section{A condition on $K$ for identifying the homotopy type of $\cxx^{K}$} 
\label{sec:condition} 

The goal of this section is to prove Theorem~\ref{conefillpo}, 
which considers conditions for building a simplicial complex $K$  
one simplex at a time in such a way that the homotopy type 
of $\cxx^{K}$ can be determined. This will be a key tool in 
proving Theorem~\ref{cxxshifted}, which identifies the homotopy 
type of $\cxx^{K}$ for a shifted complex $K$. 

We begin with a standard definition from combinatorics. Given 
simplicial complexes $K_{1}$ and $K_{2}$ on sets $\mathcal{S}_{1}$ 
and $\mathcal{S}_{2}$ respectively, the \emph{join} $K_{1}*K_{1}$ 
is the simplicial complex
\[K_{1}*K_{2}:=\{\sigma\subset\mathcal{S}_{1}\cup\mathcal{S}_{2} \,\vert\
    \sigma=\sigma_{1}\cup\sigma_{2}, \sigma_{1}\in K_{2}, \sigma_{2}\in K_{1}\}\] 
on the set $\mathcal{S}_{1}\cup\mathcal{S}_{2}$. The definition of the 
polyhedral product immediately implies the following. 

\begin{lemma} 
   \label{joinproduct} 
   Let $K_{1}$ and $K_{2}$ be simplicial complexes on index sets 
   $\{1,\ldots,n\}$ and $\{n+1,\ldots,m\}$ respectively. Then 
   $\uxa^{K_{1}\ast K_{2}}=\uxa^{K_{1}}\times\uxa^{K_{2}}$.~$\qqed$ 
\end{lemma}

For example, if $K$ is a simplicial complex on the index set $[n]$ then 
$K\ast\{n+1\}$ is the cone on $K$. Applying Lemma~\ref{joinproduct} 
we obtain the following, which will be of use later.  

\begin{corollary} 
   \label{conetype} 
   Let $K$ be a simplicial complex on the index set $[n]$. Then  
   $\uxa^{K\ast\{n+1\}}=\uxa^{K}\times X_{n+1}$. Consequently, 
   we obtain $\cxx^{K\ast\{n+1\}}=\cxx^{K}\times CX_{n+1}$.~$\qqed$  
\end{corollary} 

As another way in which joins will be used later, consider the 
inclusions of $\Delta^{n-1}_{n-2}$  into $\Delta^{n-1}_{n-2}\ast\{n+1\}$ 
and $\Delta^{n-1}$. If $L$ is the pushout of these two inclusions, then 
checking simplices immediately shows that $L=\Delta^{n}_{n-1}$. 

\begin{lemma} 
   \label{deltapushout} 
   There is a pushout 
   \[\diagram 
          \Delta^{n-1}_{n-2}\rto\dto & \Delta^{n-1}\dto \\ 
          \Delta^{n-1}_{n-2}\ast\{n+1\}\rto & \Delta^{n}_{n-1}. 
     \enddiagram\] 
   $\qqed$ 
\end{lemma} 

Applying Proposition~\ref{uxaunion} in the case of \cxx\ to the 
pushout in Lemma~\ref{deltapushout}, we obtain a homotopy equivalence 
\begin{equation} 
  \label{coneprepeqn} 
  \cxx^{\Delta^{n}_{n-1}}\simeq\cxx^{\Delta_{n-1}^{n-2}\ast\{n+1\}} 
        \cup_{\cxx^{\overline{\Delta}^{n-1}_{n-2}}}\cxx^{\overline{\Delta}^{n-1}}. 
\end{equation}  
It will be useful to state this homotopy equivalence more explicitly. 
By~(\ref{cxxfatwedge}), 
\[\cxx^{\Delta^{n-1}_{n-2}}= 
     \bigcup_{i=1}^{n} CX_{1}\times\cdots\times X_{i}\times\cdots\times CX_{n}.\]  
So by definition of $\overline{\Delta}^{n-1}_{n-2}$ we have 
\[\cxx^{\overline{\Delta}^{n-1}_{n-2}} \left(\bigcup_{i=1}^{n} 
        CX_{1}\times\cdots\times X_{i}\times\cdots\times X\times CX_{k}\right)\times X_{n+1}.\] 
As well, by the definition of the polyhedral product, we have 
\[\cxx^{\Delta^{n-1}}=CX_{1}\times\cdots\times CX_{n}.\] 
So by definition of $\overline{\Delta}^{n-1}$ we have 
\[\cxx^{\overline{\Delta}^{n-1}}=CX_{1}\times\cdots\times CX_{n}\times X_{n+1}.\] 
By Lemma~\ref{conetype}, we have 
\[\cxx^{\Delta^{n-1}_{n-2}\ast\{n+1\}}=\cxx^{\Delta^{n-1}_{n-2}}\times CX_{n+1}.\]  
Thus we obtain 
\[\cxx^{\Delta^{n-1}_{n-2}\ast\{n+1\}}=\left(\bigcup_{i=1}^{n} 
       CX_{1}\times\cdots\times X_{i}\times\cdots\times CX_{n}\right)\times CX_{n+1}.\] 
Therefore~(\ref{coneprepeqn}) states the following. 

\begin{lemma} 
   \label{explicitconepo} 
   There is a pushout 
   \[\diagram 
            \left(\bigcup_{i=1}^{n} CX_{1}\times\cdots\times X_{i}\times 
                    \cdots\times X\times CX_{n}\right)\times X_{n+1}\rto^-{b}\dto^{a} 
             & CX_{1}\times\cdots\times CX_{n}\times X_{n+1}\dto \\ 
           \left(\bigcup_{i=1}^{n} CX_{1}\times\cdots\times X_{i}\times 
                  \cdots\times CX_{n}\right)\times CX_{n+1}\rto 
               & \cxx^{\Delta^{n}_{n-1}} 
     \enddiagram\] 
   where the maps $a$ and $b$ are coordinate-wise inclusions. $\qqed$ 
\end{lemma} 
 
Note that this pushout identifies $\cxx^{\Delta^{n}_{n-1}}$ as 
$\bigcup_{i=1}^{n+1} CX_{1}\times\cdots\times X_{i}\times\cdots\times CX_{n+1}$, 
which matches the description in~(\ref{cxxfatwedge}). Since $a$ is a 
coordinate-wise inclusion and $CX_{n+1}$ is contractible, $a$~is homotopic 
to the composite 
\[\overline{\pi}_{1}\colon\namedright{\left(\bigcup_{i=1}^{n} 
        CX_{1}\times\cdots\times X_{i}\times\cdots\times X\times CX_{n}\right)\times X_{n+1}} 
        {\pi_{1}}{\bigcup_{i=1}^{n} CX_{1}\times\cdots\times X_{i}\times\cdots\times CX_{n}}\]  
\[\hspace{5cm}\stackrel{i_{1}}{\longrightarrow} 
       \left(\bigcup_{i=1}^{n} CX_{1}\times\cdots\times X_{i}\times\cdots\times CX_{n} 
       \right)\times CX_{n+1}\] 
where $\pi_{1}$ is the projection and $i_{1}$ is the inclusion. Similarly, 
since~$b$ is a coordinate-wise inclusion and $CX_{1}\times\cdots\times CX_{n}$ 
is contractible, $b$ is homotopic to the composite 
\[\overline{\pi}_{2}\colon\nameddright{\left(\bigcup_{i=1}^{n} 
        CX_{1}\times\cdots\times X_{i}\times\cdots\times X\times CX_{n}\right)\times X_{n+1}} 
        {\pi_{2}}{X_{n+1}}{i_{2}}{CX_{1}\times\cdots\times CX_{n}\times X_{n+1}}\] 
where $\pi_{2}$ is the projection and $i_{2}$ is the inclusion. 

The pushout in Lemma~\ref{explicitconepo} and the description of 
the maps $a$ and $b$ play a key role in helping to identify the homotopy 
types of certain $\cxx^{K}$'s. The statement we are aiming for is 
Theorem~\ref{conefillpo}. We first need a preliminary lemma which identifies  
the homotopy type of a certain pushout. For a product $\prod_{i=1}^{n} X_{i}$, let 
\(\pi_{j}\colon\namedright{\prod_{i=1}^{n} X_{i}}{}{X_{j}}\) 
be the projection onto the $j^{th}$-factor. 

\begin{lemma} 
   \label{polemma} 
   Suppose there is a homotopy pushout 
   \[\diagram 
            A\times B\times C\rto^-{\pi_{2}\times\pi_{3}}\dto^{f} 
                    & B\times C\dto \\ 
            P\rto & Q   
     \enddiagram\] 
   where $f$ factors as the composite 
   \(\llnamedddright{A\times B\times C}{\pi_{1}\times\pi_{3}}{A\times C} 
        {}{A\rtimes C}{f^{\prime}}{P}\). 
   Then there is a homotopy equivalence 
   \[Q\simeq D\vee [(A\ast B)\rtimes C]\] 
   where $D$ is the cofibre of $f^{\prime}$. 
\end{lemma} 

\begin{proof} 
First we recall two general facts. First, the pushout of the projections 
\(\namedright{X\times Y}{\pi_{1}}{X}\) 
and 
\(\namedright{X\times Y}{\pi_{2}}{Y}\) 
is homotopy equivalent $X\ast Y$, and the map from each of $X$ and $Y$ 
into $X\ast Y$ is null homotopic. Second, if $Q$ is the pushout of maps 
\(\namedright{X}{a}{Y}\) 
and 
\(\namedright{X}{b}{Z}\) 
then, for any space $T$, the pushout of 
\(\namedright{X\times T}{a\times 1}{Y\times T}\) 
and 
\(\namedright{X\times T}{b\times 1}{Z\times T}\) 
is $Q\times T$. 

In our case, since $f$ factors through the projection onto $A\times C$, 
there is a diagram of iterated homotopy pushouts 
\[\diagram 
       A\times B\times C\rto^-{\pi_{2}\times\pi_{3}}\dto^{\pi_{1}\times\pi_{3}} 
             & B\times C\dto \\ 
       A\times C\rto^-{g}\dto^{\bar{f}} & R\dto \\ 
       P\rto & Q 
  \enddiagram\] 
which defines the space $R$ and the map $g$. Observe that the 
top square is the product of $C$ with the pushout of the projections 
\(\namedright{A\times B}{\pi_{1}}{A}\) 
and 
\(\namedright{A\times B}{\pi_{2}}{B}\). 
Thus $R\simeq(A\ast B)\times C$ and $g\simeq\ast\times 1$. 
The identification of $R$ and $g$ lets us write the bottom pushout above 
as a diagram of iterated homotopy pushouts 
\[\diagram 
       A\times C\rto^-{\pi_{2}}\dto & C\rto^-{1}\dto 
             & (A\ast B)\times C\dto \\ 
       P\rto & Q^{\prime}\rto & Q. 
  \enddiagram\] 
By hypothesis, the restriction of 
\(\namedright{A\times C}{}{P}\) 
to $C$ is null homotopic. Thus we can pinch out $C$ in the previous 
diagram to obtain a diagram of iterated homotopy pushouts 
\[\diagram 
       A\rtimes C\rto^{f^{\prime}}\dto & \ast\rto\dto & (A\ast B)\rtimes C\dto \\ 
       P\rto & D\rto & Q.  
  \enddiagram\]
The left pushout implies that $D$ is the homotopy cofibre of $f^{\prime}$, 
and the right pushout immediately implies that 
$Q\simeq D\vee [(A\ast B)\rtimes C]$. 
\end{proof} 

Let $K$ be a simplicial complex on the index set $[n]$ and suppose 
that $K=K_{1}\cup_{\partial\sigma}\sigma$ for some simplicial 
complex $K_{1}$ and simplex $\sigma$. We consider cases where 
the inclusion of $\sigma$ into $K_{1}$ factors through a cone on 
$\sigma$, and use this to help identify the homotopy type 
of $\cxx^{K}$. To be concrete, we need some notation. 
For a sequence $(i_{1},\ldots,i_{k})$ with $1\leq i_{1}<\cdots <i_{k}\leq n$, 
let $\Delta^{i_{1},\ldots,i_{k}}$ be the full $(k-1)$-dimensional 
simplex on the vertices $\{i_{1},\ldots,i_{k}\}$. Let 
$\Delta^{i_{1},\ldots,i_{k}}_{k-2}$ be the full $(k-2)$-skeleton 
of $\Delta^{i_{1},\ldots,i_{k}}$. To match a later application of 
Theorem~\ref{conefillpo} in Section~\ref{sec:shifted}, we will assume 
that in $K=K_{1}\cup_{\partial\sigma}\sigma$ we have 
$\sigma=(i_{1},\ldots,i_{k})$ and $i_{1}\geq 2$. Observe that the 
inclusion of~$\partial\sigma$ into $K_{1}$ induces a map 
\(\namedright{\cxx^{\partial\sigma}}{}{\cxx^{K_{1}}}\). 
By Proposition~\ref{join}, there is a homotopy equivalence 
$\cxx^{\partial\sigma}\simeq X_{i_{1}}\ast\cdots\ast X_{i_{k}}$,  
so we obtain a map 
\(\namedright{X_{i_{1}}\ast\cdots\ast X_{i_{k}}}{}{\cxx^{K_{1}}}\). 
Let $(j_{1},\ldots,j_{n-k-1})$ be the 
complement of $(1,i_{1},\ldots,i_{k})$ in $(1,\ldots,n)$, and assume that 
$j_{1}<\cdots <j_{n-k-1}$. Let $\widehat{X}=\prod_{t=1}^{n-k-1} X_{j_{t}}$. 

\begin{theorem} 
   \label{conefillpo} 
   Let $K$ be a simplicial complex on the index set $[n]$. Suppose that  
   $K=K_{1}\cup_{\partial\sigma}\sigma$ where: 
   \begin{letterlist} 
      \item for $1\leq i\leq n$, the vertex $\{i\}\in K_{1}$; 
      \item $\sigma=(i_{1},\ldots,i_{k})$ for $2\leq i_{1}<\cdots <i_{k}\leq n$; 
      \item $\sigma\notin K_{1}$;  
      \item $(1)\ast\partial\sigma\subseteq K_{1}$. 
   \end{letterlist}  
   Then there is a homotopy equivalence 
   \[\cxx^{K}\simeq D\vee 
        [((X_{i_{1}}\ast\cdots\ast X_{i_{k}})\ast X_{1})\rtimes\widehat{X}]\] 
   where $D$ is the cofibre of the map 
   \(\namedright{(X_{i_{1}}\ast\cdots\ast X_{i_{k}})\rtimes\widehat{X} 
        \simeq\cxx^{\partial\sigma}\rtimes\widehat{X}} 
          {}{\cxx^{K_{1}}}\). 
\end{theorem} 

\begin{proof} 
Since the inclusion 
\(\namedright{\partial\sigma}{}{K_{1}}\) 
factors as the composite 
\(\nameddright{\partial\sigma}{}{(1)\ast\partial\sigma}{}{K_{1}}\), 
we obtain an iterated pushout diagram 
\[\diagram 
         \partial\sigma\rto\dto & \sigma\dto \\ 
         (1)\ast\partial\sigma\rto\dto & L\dto \\ 
         K_{1}\rto & K 
  \enddiagram\] 
which defines the simplicial complex $L$. Proposition~\ref{uxaunion} 
therefore implies there is an iterated pushout diagram 
\begin{equation} 
  \label{K1Kpo} 
  \diagram 
          \cxx^{\overline{\partial\sigma}}\rto\dto 
                 & \cxx^{\overline{\sigma}}\dto \\ 
          \cxx^{\overline{(1)\ast\partial\sigma}}\rto\dto 
                 & \cxx^{\overline{L}}\dto \\ 
          \cxx^{K_{1}}\rto & \cxx^{K}. 
  \enddiagram 
\end{equation}  
where the bar over each of $\partial\sigma, \sigma, (1)\ast\sigma$ and $L$ 
means they are to be regarded as simplicial complexes on the index set $[n]$. 

By hypothesis, $\sigma=(i_{1},\ldots,i_{k})$, so 
$\sigma=\Delta^{i_{1},\ldots,i_{k}}_{k-2}$. The pushout defining $L$ 
therefore implies that $L=\Delta^{k}_{1,i_{1},\ldots,i_{k}}$. 
Now, arguing in the same way that produced the diagram in 
Lemma~\ref{explicitconepo}, an explicit description of the 
upper pushout in~(\ref{K1Kpo}) is as follows:  
\begin{equation} 
   \label{conepogeneralized} 
   \diagram 
         \left(\bigcup_{j=1}^{k} CX_{i_{1}}\times\cdots\times X_{i_{j}}\times 
                \cdots\times X\times CX_{i_{k}}\right)\times X_{1}\times\widehat{X} 
                \rto^-{b}\dto^{a} 
           & CX_{i_{1}}\times\cdots\times CX_{i_{k}}\times X_{1}\times\widehat{X}\dto \\ 
        \left(\bigcup_{j=1}^{k} CX_{i_{1}}\times\cdots\times X_{i_{j}}\times 
               \cdots\times CX_{i_{k}}\right)\times CX_{1}\times\widehat{X}\rto 
            & \cxx^{\Delta^{1,i_{1},\ldots,i_{k}}_{k-1}}\times\widehat{X}.  
  \enddiagram 
\end{equation} 
Note that, rearranging the indices, (\ref{conepogeneralized})~ is just the 
product of a pushout as in Lemma~\ref{explicitconepo} with $\widehat{X}$. 
As well, as noted after Lemma~\ref{explicitconepo}, up to homotopy, 
$a$ factors through the projection onto 
$\left(\bigcup_{j=1}^{k} CX_{i_{1}}\times\cdots\times X_{i_{j}}\times 
                \cdots\times X\times CX_{i_{k}})\right)\times\widehat{X}$ 
and $b$ factors through the projection onto $X_{1}\times\widehat{X}$.  
By Proposition~\ref{join}, there are homotopy equivalences 
\[\left(\bigcup_{j=1}^{k} CX_{i_{1}}\times\cdots\times X_{i_{j}}\times 
        \cdots\times X\times CX_{i_{k}}\right)\simeq X_{i_{1}}\ast\cdots\ast X_{i_{k}}\]  
and 
\[\cxx^{\Delta^{1,i_{1},\ldots,i_{k}}_{k-2}}\simeq 
         X_{1}\ast X_{i_{1}}\ast\cdots\ast X_{i_{k}}.\] 
Thus, up to homotopy equivalences, (\ref{conepogeneralized})~is equivalent 
to the homotopy pushout 
\[\diagram 
         (X_{i_{1}}\ast\cdots\ast X_{i_{k}})\times X_{1}\times\widehat{X} 
                  \rto^-{\mbox{proj}}\dto^{\mbox{proj}} 
             & X_{1}\times\widehat{X}\dto \\ 
        (X_{i_{1}}\ast\cdots\ast X_{i_{k}})\times\widehat{X}\rto 
             & (X_{1}\ast X_{i_{1}}\ast\cdots\ast X_{i_{k}})\times\widehat{X}. 
  \enddiagram\] 
Therefore, up to homotopy equivalences, (\ref{K1Kpo})~is equivalent 
to the iterated homotopy pushout diagram 
\[\diagram 
         (X_{i_{1}}\ast\cdots\ast X_{i_{k}})\times X_{1}\times\widehat{X} 
                  \rto^-{\mbox{proj}}\dto^{\mbox{proj}} 
             & X_{1}\times\widehat{X}\dto \\ 
        (X_{i_{1}}\ast\cdots\ast X_{i_{k}})\times\widehat{X}\rto\dto 
             & (X_{1}\ast X_{i_{1}}\ast\cdots\ast X_{i_{k}})\times\widehat{X}\dto \\ 
        \cxx^{K_{1}}\rto & \cxx^{K}.  
  \enddiagram\] 
By hypothesis, each vertex $\{i\}\in K_{1}$ for $1\leq i\leq n$, 
so Proposition~\ref{projincl} implies that the restriction of  
\(\namedright{(X_{i_{1}}\ast\cdots\ast X_{i_{k}})\times\widehat{X}}{}{\cxx^{K_{1}}}\) 
to $\widehat{X}$ is null homotopic. Thus the outer perimeter of the previous 
diagram is a homotopy pushout 
\[\diagram 
         (X_{i_{1}}\ast\cdots\ast X_{i_{k}})\times X_{1}\times\widehat{X} 
                  \rto^-{\mbox{proj}}\dto^{f} 
             & X_{1}\times\widehat{X}\dto \\ 
        \cxx^{K_{1}}\rto & \cxx^{K} 
  \enddiagram\] 
where $f$ factors as the composite 
\(\llnamedddright{(X_{i_{1}}\ast\cdots\ast X_{i_{k}})\times X_{1}\times\widehat{X}} 
       {\pi_{1}\times\pi_{3}}{(X_{i_{1}}\ast\cdots\ast X_{i_{k}})\times\widehat{X}} 
       {}{(X_{i_{1}}\ast\cdots\ast X_{i_{k}})\rtimes\widehat{X}}{f^{\prime}}{\cxx^{K_{1}}}\). 
Lemma~\ref{polemma} therefore implies that 
\[\cxx^{K}\simeq D\vee 
     [((X_{i_{1}}\ast\cdots\ast X_{i_{k}})\ast X_{1})\rtimes\widehat{X}],\] 
where $D$ is the cofiber of the map 
   \(\namedright{(X_{i_{1}}\ast\cdots\ast X_{i_{k}})\rtimes\widehat{X} 
        \simeq\cxx^{\partial\sigma}\rtimes\widehat{X}} 
          {f^{\prime}}{\cxx^{K_{1}}}\). 
\end{proof}

\section{Polyhedral products for shifted complexes} 
\label{sec:shifted} 

In this section we prove Theorem~\ref{cxxshifted}. To begin, 
we introduce some definitions which are standard in combinatorics. 

\begin{definition} 
Let $K$ be a simplicial complex on $n$ vertices. The complex $K$ is 
\emph{shifted} if there is an ordering on its vertices such that whenever 
$\sigma\in K$ and $\nu^{\prime}<\nu$, then 
$(\sigma-\nu)\cup\nu^{\prime}\in K$. 
\end{definition} 

It may be helpful to interpret this definition in terms of ordered sequences. 
Let $K$ be a simplicial complex on $[n]$ and order the vertices by their 
integer labels. If $\sigma\in K$ with vertices $\{i_{1},\ldots,i_{k}\}$ 
where $1\leq i_{1}<\cdots,i_{k}\leq n$, then regard $\sigma$ as the 
ordered sequence $(i_{1},\ldots,i_{k})$.  The shifted condition 
states that if $\sigma=(i_{1},\ldots,_{k})\in K$ then $K$ contains 
every simplex $(t_{1},\ldots,t_{l})$ with $l\leq k$ and 
$t_{1}\leq i_{1},\ldots,t_{l}\leq i_{l}$. 

\begin{examples} 
   \label{shiftedexamples}  
   We give three examples. 
   \begin{enumerate} 
      \item[(1)] Let $K$ be the simplicial complex with vertices $\{1,2,3,4\}$ 
                      and edges $\{(1,2),(1,3),(1,4),(2,4)\}$. That is, $K$ is two 
                      copies of $\Delta^{2}_{1}$ glued along a common edge. 
                      Then $K$ is shifted. 
      \item[(2)] Let $K$ be the simplicial complex with vertices $\{1,2,3,4\}$ 
                      and edges $\{(1,2),(2,3),(3,4),(1,4)\}$. That is, $K$ is the 
                      boundary of a square. Then $K$ is not shifted. 
      \item[(3)] For $0\leq k\leq n-1$, the full $k$-skeleton of $\Delta^{n}$ 
                      is shifted. 
   \end{enumerate} 
\end{examples}  

\begin{definition} 
Let $K$ be a simplicial complex. The \emph{rest}, \emph{star} and \emph{link} 
of a simplex $\sigma\in K$ are the subcomplexes 
\[\begin{array}{lcl} 
     \starr_{K}\sigma & = & \{\tau\in K\mid\sigma\cup\tau\in K\}; \\ 
      \rest_{K}\sigma & = 
          & \{\tau\in K\mid\sigma\cap\tau=\emptyset\}; \\ 
     \link_{K}\sigma & = & \starr_{K}\sigma\cap\rest_{K}\sigma. 
   \end{array}\] 
\end{definition} 

There are three standard facts that follow straight from the definitions. 
First, there is a pushout 
\[\diagram 
         \link_{K}\sigma\rto\dto & \rest_{K}\sigma\dto \\ 
         \starr_{K}\sigma\rto & K. 
  \enddiagram\] 
Second, if $K$ is shifted then so are $\rest_{K}\sigma$, $\starr_{K}\sigma$ 
and $\link_{K}\sigma$ for each $\sigma\in K$. Third, $\starr_{K}\sigma$ 
is a join: $\starr_{K}\sigma=\sigma\ast\link_{K}\sigma$. 

For $K$ a simplicial complex on $[n]$ and $\sigma$ being a 
vertex $(i)$, we write $\rest\{1,\ldots,\hat{i},\ldots,n\}$, 
$\starr(i)$ and $\link(i)$ for $\rest_{K}\sigma$, $\starr_{K}\sigma$ 
and $\link_{K}\sigma$. To illustrate, take $i_{1}=1$. Then 
$\starr(1)$ consists of those simplices $(i_{1},\ldots,i_{k})\in K$ 
with $i_{1}=1$; $\rest\{2,\ldots,n\}$ consists 
of those simplices $(j_{1},\ldots,j_{k})\in K$ with $j_{1}>1$, 
and $\link(1)=\starr(1)\cap\rest\{2,\ldots,n\}$. The three useful 
facts mentioned above become the following. First, there is a pushout  
\[\diagram 
       \link(1)\rto\dto & \rest\{2,\ldots,n\}\dto \\ 
       \starr(1)\rto & K. 
  \enddiagram\] 
Second, if $K$ is shifted then so are $\rest\{2,\ldots,n\}$, 
$\starr(1)$ and $\link(1)$. Third, $\starr(1)$ is a join: 
$\starr(1)=(1)\ast\link(1)$. 

Next, we require four lemmas to prepare for the proof 
of Theorem~\ref{cxxshifted}. The first two are about shifted complexes, 
and the next two are about decompositions. 

\begin{lemma} 
   \label{shifteddelta} 
   Let $K$ be a shifted complex on the index set $[n]$. 
   If $\sigma\in\rest\{2,\ldots,n\}$, 
   then $\partial\sigma\in\link(1)$. 
\end{lemma} 

\begin{proof} 
Suppose the ordered sequence corresponding to $\sigma$ is 
$(i_{1},\ldots,i_{k})$. Then $\partial\sigma=\bigcup_{j=1}^{k}\sigma_{j}$ 
for $\sigma_{j}=(i_{1},\ldots,\hat{i}_{j},\ldots,i_{k})$, where $\hat{i}_{j}$ 
means omit the $j^{th}$-coordinate. So to prove the lemma it is equivalent  
to show that $\sigma_{j}=(i_{1},\ldots,\hat{i}_{j},\ldots,i_{k})\in\link(1)$ for 
each $1\leq j\leq k$. 

Fix $j$. Observe that $\sigma_{j}=(i_{1},\ldots,\hat{i}_{j},\ldots,i_{k})$   
is a sequence of length $k-1$ and $2\leq i_{1}<\ldots<i_{k}\leq n$. 
We claim that the sequence $(1,i_{1},\ldots,\hat{i}_{j},\ldots,i_{k})$ of 
length $k$ represents a face of $K$. This holds because, as ordered sequences, 
we have  $(1,i_{1},\ldots,\hat{i}_{j},\ldots,i_{k})<(i_{1},\ldots,i_{k})$, and 
the shifted property for $K$ implies that as $(i_{1},\ldots,i_{k})\in K$, any 
ordered sequence less than $(i_{1},\ldots,i_{k})$ also represents a face of $K$. 
Now, as $(1,i_{1},\ldots,\hat{i}_{j},\ldots,i_{k})\in K$, we clearly have 
$(1,i_{1},\ldots,\hat{i}_{j},\ldots,i_{k})\in\starr(1)$. 
Thus the sub-simplex $(i_{1},\ldots,\hat{i}_{j},\ldots,i_{k})$ 
in also in $\starr(1)$. That is, $\sigma_{j}\in\starr(1)$. Hence 
$\sigma_{j}\in\starr(1)\cap\rest\{2,\ldots,n\}=\link(1)$, as required.   
\end{proof} 

\begin{remark} 
In Lemma~\ref{shifteddelta}, it may be that $\sigma$ 
itself is in $\link(1)$, but this need not be the case. For if we argue as 
in the proof of the lemma, we obtain $\sigma\in\link(1)$ if and 
only if $(1,i_{1},\ldots,i_{k})\in K$. 
\end{remark} 

\begin{remark} 
\label{coneremark} 
It is also worth noting that as $\partial\sigma\in\link(1)$ and 
$\starr(1)=(1)\ast\link(1)$, we have $(1)\ast\partial\sigma\subseteq\starr(1)$. 
That is, the cone on $\partial\sigma$ is in $\starr(1)$. 
\end{remark} 

We say that a face $\tau$ of a simplicial complex $K$ is \emph{maximal} 
if there is no other face $\tau^{\prime}\in K$ with~$\tau$ a proper 
subset of $\tau^{\prime}$. 

\begin{lemma} 
   \label{shiftedfilt} 
   Let $K$ be a shifted complex on the index set $[n]$. Then the map 
   \(\namedright{\link(1)}{}{\rest\{2,\ldots,n\}}\) 
   is filtered by a sequence of simplicial complexes 
   \[\link(1)=L_{0}\subseteq L_{1}\subseteq\cdots\subseteq  
          L_{m}=\rest\{2,\ldots,n\}\] 
   where $L_{i} =L_{i-1}\cup\tau_{i}$ and $\tau_{i}$ satisfies: 
   \begin{letterlist} 
      \item $\tau_{i}$ is maximal in $\rest\{2,\ldots,n\}$; 
      \item $\tau_{i}\notin\link(1)$; 
      \item $\partial\tau_{i}\in\link(1)$. 
   \end{letterlist} 
\end{lemma} 

\begin{proof} 
In general, if $L$ is a connected simplicial complex and 
$L_{0}\subseteq L$ is a subcomplex (not necessarily connected), 
it is possible to start with $L_{0}$ and sequentially glue in faces 
one at a time to get $L$. That is, there is a sequence of simplicial 
complexes $L_{0}=\subseteq L_{1}\subseteq\cdots\subseteq L_{m}=L$ 
where $L_{i}=L_{i-1}\cup\tau_{i}$ for some simplex $\tau_{i}\in L$, 
$\tau_{i}\notin L_{i-1}$ and the union is taken over the boundary 
$\partial\tau_{i}$ of $\tau_{i}$. In addition, it may be assumed that 
the adjoined faces $\tau_{i}$ are maximal in $L$. Thus parts~(a) 
and~(b) of the lemma follow. For part~(c), since $K$ is shifted and 
each $\tau_{i}\in\rest\{2,\ldots,n\}$, Lemma~\ref{shifteddelta} 
implies that $\partial\sigma\in\link(1)$. 
\end{proof}  

Next, we turn to the two decomposition lemmas. 

\begin{lemma} 
   \label{suspprod} 
   For any spaces $M, N_{1},\ldots,N_{m}$, there is a homotopy equivalence 
   \[\Sigma M\rtimes (N_{1}\times\cdots\times N_{m})\simeq\Sigma M\vee 
         \left(\bigvee_{1\leq t_{1}<\cdots<t_{k}\leq m} 
         \Sigma M\wedge N_{i_{1}}\wedge\cdots\wedge N_{i_{k}}\right).\] 
\end{lemma} 

\begin{proof} 
In general, $\Sigma X\rtimes Y\simeq\Sigma X\vee(\Sigma X\wedge Y)$, 
so it suffices to decompose $\Sigma M\wedge (N_{1}\times\cdots\times N_{m})$. 
Iterating the basic fact that 
$\Sigma(X\times Y)\simeq\Sigma X\vee\Sigma Y\vee(\Sigma X\wedge Y)$, 
we obtain a homotopy equivalence 
$\Sigma(N_{1}\times\cdots\times N_{m})\simeq 
      \bigvee_{1\leq t_{1}<\cdots<t_{k}\leq m} 
      (\Sigma N_{i_{1}}\wedge\cdots\wedge N_{i_{k}})$. 
Thus, as $X\ast Y\simeq\Sigma X\wedge Y\simeq X\wedge\Sigma Y$ for any 
space $X$, we have 
\begin{eqnarray*} 
    M\ast(N_{1}\times\cdots\times N_{m}) & \simeq 
      & M\wedge\Sigma(N_{1}\times\cdots\times N_{m}) \\ 
    & \simeq & M\wedge\left(\bigvee_{1\leq t_{1}<\cdots<t_{k}\leq m} 
        \Sigma N_{i_{1}}\wedge\cdots\wedge N_{i_{k}}\right) \\ 
    & \simeq & \bigvee_{1\leq t_{1}<\cdots<t_{k}\leq m} 
        M\wedge\Sigma N_{i_{1}}\wedge\cdots\wedge N_{i_{k}}. 
  \end{eqnarray*} 
\end{proof} 
      
Recall from Section~\ref{sec:condition} 
that if $K$ is a simplicial complex on the index set $[n]$ 
then $\Delta^{i_{1},\ldots,i_{k}}$ is the full $(k-1)$-dimensional 
simplex on the vertex set $\{i_{1},\ldots,i_{k}\}$ for 
$1\leq i_{1}<\cdots<i_{k}\leq n$, and 
$\Delta^{i_{1},\ldots,i_{k}}_{k-2}$ is the full $(k-2)$-skeleton 
of~$\Delta^{i_{1},\ldots,i_{k}}$. 

\begin{lemma} 
   \label{fatwedgeretract} 
   Let $K$ be a simplicial complex on the index set $[n]$. Suppose for 
   some sequence $1\leq i_{1}<\cdots <i_{k}\leq n$, we have 
   $\Delta^{i_{1},\ldots,i_{k}}_{k-2}\subseteq K$ but 
   $\Delta^{i_{1},\ldots,i_{k}}\subsetneqq K$. Then the map 
   \(\namedright{\cxx^{\Delta^{i_{1},\ldots,i_{k}}_{k-2}}}{}{\cxxk}\) 
   induced by the inclusion 
   \(\namedright{\Delta^{i_{1},\ldots,i_{k}}_{k-2}}{}{K}\) 
   has a left inverse. Consequently,  
   $X_{i_{1}}\ast\cdots\ast X_{i_{k}}$ is a retract of $\cxxk$. 
\end{lemma} 

\begin{proof} 
Since $\Delta^{i_{1},\ldots,i_{k}}_{k-2}\subseteq K$ but 
$\Delta^{i_{1},\ldots,i_{k}}\subsetneqq K$, projecting the index set $[n]$ 
onto the index set $\{i_{1},\ldots,i_{k}\}$ induces a projection 
\(\namedright{K}{}{\Delta^{i_{1},\ldots,i_{k}}_{k-2}}\). 
Observe that the composite  
\(\nameddright{\Delta^{i_{1},\ldots,i_{k}}_{k-2}}{}{K}{} 
     {\Delta^{i_{1},\ldots,i_{k}}_{k-2}}\) 
is the identity map. Thus the induced composite 
\(\nameddright{\cxx^{\Delta^{i_{1},\ldots,i_{k}}_{k-2}}}{}{\cxxk} 
     {}{\cxx^{\Delta^{i_{1},\ldots,i_{k}}_{k-2}}}\) 
is the identity map. Consequently, the homotopy equivalence 
$\cxx^{\Delta_{i_{1},\ldots,i_{k}}^{k-1}}\simeq X_{i_{1}}\ast\cdots\ast X_{i_{k}}$ 
of Proposition~\ref{join} implies that $X_{i_{1}}\ast\cdots\ast X_{i_{k}}$ 
is a retract of $\cxxk$.  
\end{proof}  

We expand on Lemma~\ref{fatwedgeretract}. Let $\{j_{1},\ldots,j_{n-k}\}$  
be the vertices in $[n]$ which are complementary to $\{i_{1},\ldots,i_{k}\}$. 
Let $\widehat{X}=\prod_{t=1}^{n-k} X_{j_{t}}$. Since the index sets 
$\{i_{1},\ldots,i_{k}\}$ and $\{j_{1},\ldots,j_{n-k}\}$ are complementary 
the inclusion 
\(\namedright{\Delta^{i_{1},\ldots,i_{k}}_{k-2}}{}{K}\) 
induces an inclusion 
\(I\colon\namedright{\cxx^{\Delta^{i_{1},\ldots,i_{k}}_{k-2}}\times\widehat{X}} 
     {}{\cxx^{K}}\). 
If each vertex of $[n]$ is in $K$, Proposition~\ref{projincl} implies that 
the restriction of $I$ to $\widehat{X}$ is null homotopic. Thus $I$ factors 
through a map 
\(I^{\prime}\colon\namedright 
     {\cxx^{\Delta^{i_{1},\ldots,i_{k}}_{k-2}}\rtimes\widehat{X}}{}{\cxx^{K}}\). 

\begin{lemma} 
   \label{fatwedgeretract2} 
   Let $K$ be a simplicial complex on the index set $[n]$ for which each 
   vertex is in $K$. Suppose for some sequence $1\leq i_{1}<\cdots <i_{k}\leq n$, 
   we have $\Delta^{i_{1},\ldots,i_{k}}_{k-2}\subseteq K$ but 
   $\Delta^{i_{1},\ldots,i_{k}}\subsetneqq K$. Then the map 
   \(I^{\prime}\colon\namedright 
         {\cxx^{\Delta^{i_{1},\ldots,i_{k}}_{k-2}}\rtimes\widehat{X}}{}{\cxxk}\) 
   induced by the inclusion 
   \(\namedright{\Delta^{i_{1},\ldots,i_{k}}_{k-2}}{}{K}\) 
   has a left inverse after suspending. 
\end{lemma} 

\begin{proof} 
The hypothesis that $\Delta^{i_{1},\ldots,i_{k}}_{k-2}\subseteq K$ but 
$\Delta^{i_{1},\ldots,i_{k}}\subsetneqq K$ implies that $(i_{1},\ldots,i_{k})$ 
is a (minimal) missing face of $K$. Thus any other sequence 
$(t_{1},\ldots,t_{l})$ for $l\leq n$ with $(i_{1},\ldots,i_{k})$ a subsequence 
also represents a missing face of $K$. The decomposition in~(\ref{BBCGdecomp}) 
therefore implies that 
$\Sigma(\Sigma^{k-1} X_{t_{1}}\wedge\cdots\wedge X_{t_{l}})$ is a summand 
of $\Sigma\cxx^{K}$. Notice that any sequence $(t_{1},\ldots,t_{l})$ 
is obtained by starting with $(i_{1},\ldots,i_{k})$ and inserting $l-k$ 
distinct element from the complement $\{j_{1},\ldots,j_{n-k}\}$ of 
$\{i_{1},\ldots,i_{k}\}$ in $[n]$. Thus the list of all the wedge summands 
$\Sigma(\Sigma^{k-1} X_{t_{1}}\wedge\cdots\wedge X_{t_{l}})$ is in 
one-to-one correspondence with the wedge summands of  
$\Sigma(\Sigma^{k-1} X_{i_{1}}\wedge\cdots\wedge X_{i_{k}})\rtimes\widehat{X}$ 
after applying Lemma~\ref{suspprod}. Hence 
$\Sigma(\Sigma^{k-1} X_{i_{1}}\wedge\cdots\wedge X_{i_{k}})\rtimes\widehat{X}$ 
retracts off $\Sigma\cxx^{K}$. 
\end{proof} 

We are now ready to prove the main result in the paper. 
For convenience, let $\W_{n}$ be the collection of spaces which are 
either contractible or homotopy equivalent to a wedge of spaces of 
the form $\Sigma^{j} X_{i_{1}}\wedge\cdots\wedge X_{i_{k}}$ 
for $j\geq 1$ and $1\leq i_{1}<\cdots <i_{k}\leq n$. Note that 
for each $n>1$, $\W_{n-1}\subseteq\W_{n}$. 

\begin{proof}[Proof of Theorem~\ref{cxxshifted}] 
The proof is by induction on the number of vertices. If $n=1$ 
then $K=\{1\}$, which is shifted, and the definition of the polyhedral 
product implies that $\cxxk=CX$, which is contractible. Thus 
$K\in\W_{1}$. 

Assume the theorem holds for all shifted complexes on $k$ vertices, 
with $k<n$. Let $K$ be a shifted complex on the index set $[n]$. Consider 
the pushout 
\[\diagram 
       \link(1)\rto\dto & \rest\{2,\ldots,n\}\dto \\ 
       \starr(1)\rto & K  
  \enddiagram\] 
and recall that $\starr(1)=(1)\ast\link(1)$. Since $K$ is shifted, so are 
$\rest\{2,\ldots,n\}$, $\starr(1)$ and $\link(1)$. Note that 
$\rest\{2,\ldots,n\}$ is a shifted complex on $n-1$ vertices, 
and as $\link(1)$ is a subcomplex of $\rest\{2,\ldots,n\}$, it too 
is a shifted complex on $n-1$ vertices. Therefore, by inductive hypothesis, 
$\cxx^{\link(1)}\in\W_{n-1}$.  

By Lemma~\ref{shiftedfilt}, the map 
\(\namedright{\link(1)}{}{\rest\{2,\ldots,n\}}\) 
is filtered by a sequence of simplicial complexes 
\[\link(1)=L_{0}\subseteq L_{1}\subseteq\cdots\subseteq  
        L_{m}=\rest\{2,\ldots,n\}\] 
where $L_{i} =L_{i-1}\cup\tau_{i}$ and $\tau_{i}$ satisfies: 
(i)~$\tau_{i}$ is maximal in $\rest\{2,\ldots,n\}$; (ii)~$\tau_{i}\notin\link(1)$; 
and (iii)~$\partial\tau_{i}\in\link(1)$. In particular, for each $1\leq i\leq m$, 
there is a pushout 
\begin{equation} 
   \label{Lpo} 
   \diagram 
       \partial\tau_{i}\rto\dto & \tau_{i}\dto \\ 
       L_{i-1}\rto & L_{i}. 
  \enddiagram 
\end{equation} 
Let $K_{0}=\starr(1)$, and for 
$1\leq i\leq m$, define $K_{i}$ as the simplicial complex obtained from 
the pushout 
\begin{equation} 
  \label{Kpo} 
  \diagram 
         L_{i-1}\rto\dto & L_{i}\dto \\ 
         K_{i-1}\rto & K_{i}. 
  \enddiagram 
\end{equation} 
Observe that we obtain a filtration of the map 
\(\namedright{\starr(1)}{}{K}\) 
as a sequence 
$\starr(1)=K_{0}\subseteq K_{1}\subseteq\cdots\subseteq K_{m}=\rest\{2,\ldots,n\}$. 
Juxtaposing the pushouts in~(\ref{Lpo}) and~(\ref{Kpo}) we obtain a pushout 
\begin{equation} 
  \label{KLpo} 
  \diagram 
       \partial\tau_{i}\rto\dto & \tau_{i}\dto \\ 
       K_{i-1}\rto & K_{i}. 
  \enddiagram 
\end{equation} 

Since $\partial\tau_{i}\in\link(1)$, Remark~\ref{coneremark} 
implies that $(1)\ast\partial\tau_{i}\in\starr(1)$. Thus as 
$\starr(1)=K_{0}$, the map 
\(\namedright{\partial\tau_{i}}{}{K_{i-1}}\) 
factors as the composite 
\(\namedddright{\partial\tau_{i}}{}{(1)\ast\partial\tau_{i}}{}{\starr(1)=K_{0}} 
      {}{K_{i-1}}\). 
That is, the inclusion of $\partial\tau_{i}$ into $K_{i-1}$ 
factors through the cone on $\partial\tau_{i}$. 

We now argue that each $\cxx^{K_{j}}\in\W_{n}$. First consider $\cxx^{K_{0}}$. 
Since $K_{0}=\starr(1)=(1)\ast\link(1)$, by Lemma~\ref{joinproduct} we have 
$\cxx^{K_{0}}=\cxx^{(1)}\times\cxx^{\link(1)}$. By the definition of 
the polyhedral product, $\cxx^{(1)}=CX_{1}$, so $\cxx^{K_{0}}$ is 
homotopy equivalent to $\cxx^{\link(1)}$. By inductive hypothesis, 
$\cxx^{\link(1)}\in\W_{n-1}$. Thus $\cxx^{K_{0}}\in\W_{n}$.  

Next, fix an integer $j$ such that $1\leq j\leq m$, and assume 
that $\cxx^{K_{j-1}}\in\W_{n}$. We have 
$K_{j}=K_{j-1}\cup_{\partial\tau_{j}}\tau_{j}$. 
Since $\tau_{1}=\Delta^{i_{1},\ldots,i_{k}}$ for some sequence 
$(i_{1},\ldots,i_{k})$, we have $\partial\tau_{1}=\Delta^{i_{1},\ldots,i_{k}}_{k-2}$. 
Let $(j_{1},\ldots,j_{n-k-1})$ be the complement of 
$(1,i_{1},\ldots,i_{k})$ in $[n]$, and let 
$\widehat{X}=\prod_{t=1}^{n-k-1} X_{j_{t}}$. Since 
\(\namedright{\partial\tau_{j}}{}{K_{j-1}}\) 
factors through $(1)\ast\partial\tau_{j}$,  by Theorem~\ref{conefillpo} 
there is a homotopy equivalence 
\[\cxx^{K_{j}}\simeq D_{j}\vee 
      \left(((X_{i_{1}}\ast\cdots\ast X_{i_{k}})\ast X_{1})\rtimes\widehat{X}\right)\] 
where $D_{j}$ is the cofiber of the map 
\(\namedright{(X_{i_{1}}\ast\cdots\ast X_{i_{k}})\rtimes\widehat{X} 
       \simeq\cxx^{\partial\tau_{j}}\rtimes\widehat{X}}{}{\cxx^{K_{j-1}}}\). 
Since $\widehat{X}$ is a product, if we take $M=X_{i_{1}}\ast\cdots\ast X_{i_{k}}\ast X_{1}$ 
then Lemma~\ref{suspprod} implies that    
$((X_{i_{1}}\ast\cdots\ast X_{i_{k}})\ast X_{1})\rtimes\widehat{X}\in\W_{n}$. 
If $D_{j}\in\W_{n}$ as well, then $\cxx^{K_{j}}\in\W_{n}$. Therefore, by induction,  
$\cxx^{K_{m}}\in\W_{n}$. But $\cxx^{K}=\cxx^{K_{m}}$, so $\cxx^{K}\in\W_{n}$, 
which completes the inductive step on the number of vertices and therefore 
proves the theorem. 

It remains to show that $D_{j}\in\W_{n}$. Consider the cofibration 
\[\nameddright{(X_{i_{1}}\ast\cdots\ast X_{i_{k}})\rtimes\widehat{X} 
     \simeq\cxx^{\partial\tau_{j}}\rtimes\widehat{X}}{f}{\cxx^{K_{j-1}}}{}{D_{j}}.\]  
Since $\tau_{j}\notin K_{j-1}$, Lemma~\ref{fatwedgeretract2} implies that 
$\Sigma f$ has a left homotopy inverse. We claim that this implies that $f$ 
has a left homotopy inverse. By Lemma~\ref{suspprod}, 
$(X_{i_{1}}\ast\cdots\ast X_{i_{k}})\rtimes\widehat{X}\in\W_{n}$, 
and by inductive hypothesis, $\cxx^{K_{j-1}}\in\W_{n}$. Thus both 
of these spaces are homotopy equivalent to a wedge of spaces of the form 
$\Sigma^{j} X_{t_{1}}\wedge\cdots\wedge X_{t_{l}}$ for various $j$ and 
sequences $(t_{1},\ldots,t_{l})$ with $l\leq n$. Observe that each space 
$\Sigma^{j} X_{t_{1}}\wedge\cdots\wedge X_{t_{l}}$
is coordinate-wise indecomposable (that is, 
$\Sigma^{j} X_{t_{1}}\wedge\cdots\wedge X_{t_{l}}$ does not decompose 
as a wedge of spaces $\Sigma^{j} X_{u_{1}}\wedge\cdots\wedge X_{u_{l^{\prime}}}$ 
for various sequences $(u_{1},\ldots,u_{l^{\prime}})$.) 
Thus~$f$ maps a wedge of coordinate-wise indecomposable spaces 
into another such wedge. As $f$ respects the coordinate indices, 
the left homotopy inverse for $\Sigma f$ implies that $f$ has a left 
homotopy inverse. 

Now, since $\cxx^{K_{j-1}}\in\W_{n}$ (and is nontrivial), it is a suspension. 
Therefore, the existence of a left homotopy inverse for $f$ implies 
that there is a homotopy equivalence 
$\cxx^{K_{j-1}}\simeq (X_{i_{1}}\ast\cdots\ast X_{i_{k}})\vee D_{j}$. 
Thus $D_{j}$ is a retract of a space in $\W_{n}$, implying that $D_{j}\in\W_{n}$. 

Finally, at this point we have shown that $\cxx^{K}\in\W_{n}$, so 
$\cxx^{K}\simeq\bigvee_{\mathcal{J}} 
     \Sigma^{j} X_{i_{1}}\wedge\cdots\wedge X_{i_{k}}$ 
for some index set $J$. We need to show that the list of spaces in this 
wedge decomposition matches the list in the statement of the theorem,  
$\cxx^{K}\simeq\left(\bigvee_{I\notin K}\vert K_{I}\vert\ast\widehat{X}^{I}\right)$. 
But after suspending, by~\cite{BBCG1} there is a homotopy equivalence 
$\Sigma\cxx^{K}\simeq\Sigma\left( 
    \bigvee_{I\notin K}\vert K_{I}\vert\ast\widehat{X}^{I}\right)$, 
so we obtain 
\begin{equation} 
  \label{listequiv} 
  \Sigma\left(\bigvee_{\mathcal{J}}\Sigma^{j} X_{i_{1}}\wedge\cdots\wedge X_{i_{k}}\right)
    \simeq\Sigma\left(\bigvee_{I\notin K}\vert K_{I}\vert\ast\widehat{X}^{I}\right). 
\end{equation}  
The wedge summands $X_{i_{1}}\wedge\cdots\wedge X_{i_{k}}$ are 
indecomposable in a coordinate-wise sense - that is, 
$X_{i_{1}}\wedge\cdots\wedge X_{i_{k}}$ 
is not homotopy equivalent to a wedge 
$(X_{j_{1}}\wedge\cdots\wedge X_{j_{l}})\vee 
    (X_{j^{\prime}_{1}}\wedge\cdots\wedge X_{j^{\prime}_{l^{\prime}}})$. 
Therefore the wedge summands that appear on each side of the equivalence 
in~(\ref{listequiv}) must be the same and appear with the same multiplicity. 
Thus the two indexing sets in~(\ref{listequiv}) coincide, so we obtain 
$\cxx^{K}\simeq\left(\bigvee_{I\notin K}\vert K_{I}\vert\ast\widehat{X}^{I}\right)$, 
as required. 
\end{proof}

\section{Examples} 
\label{sec:examples} 

We consider the two shifted cases from Examples~(\ref{shiftedexamples}). 
First, let $K=\Delta^{n-1}_{k}$, the full $k$-skeleton of $\Delta^{n-1}$. 
Phrased in terms of polyhedral products, Porter~\cite{P2} showed that for 
any simply-connected spaces $X_{1},\ldots,X_{n}$, there is a homotopy equivalence 
\[(\underline{C\Omega X},\underline{\Omega X})^{K}\simeq\bigvee_{j=k+2}^{n}
      \left(\bigvee_{1\leq i_{1}<\cdots<i_{j}\leq n}\ \binom{j-1}{k+1}
      \Sigma^{k+1}\Omega X_{i_{1}}\wedge\cdots\wedge\Omega X_{i_{j}}\right).\] 
Theorem~\ref{cxxshifted} now generalizes this. If $X_{1},\ldots,X_{n}$ are 
any path-connected spaces, there is a homotopy equivalence 
\[\cxx^{K}\simeq\bigvee_{j=k+2}^{n}
      \left(\bigvee_{1\leq i_{1}<\cdots<i_{j}\leq n}\ \binom{j-1}{k+1}
      \Sigma^{k+1} X_{i_{1}}\wedge\cdots\wedge X_{i_{j}}\right).\]  
For example, this decomposition holds not just for $X_{i}=\Omega S^{n_{i}}$ 
as in Porter's case, but also for the spheres themselves, $X_{i}=S^{n_{i}}$. 

Second, let $K$ be the simplicial complex in Examples~\ref{shiftedexamples}~(1), 
which is two copies of $\Delta^{2}_{1}$ glued along a common edge.  
Specifically, $K$ is the simplicial complex with vertices $\{1,2,3,4\}$ and edges 
$\{(1,2),(1,3),(1,4),(2,4)\}$. To illustrate the algorithmic nature of the proof of 
Theorem~\ref{cxxshifted}, we will carry out the iterative procedure 
for identifying the homotopy type of $\cxx^{K}$. 
Starting with $K_{0}=\starr(1)$, we glue in one edge at a time: 
let $K_{1}=K_{0}\cup_{\{2,3\}}(2,3)$ and $K_{2}=K_{1}\cup_{\{2,4\}}(2,4)$. 
Note that $K_{2}=K$. We begin to identify homotopy types. 

\noindent\textit{Step 1}: For $K_{0}$ we have $\starr(1)=(1)\ast\link(1)$ 
where $\link(1)=\{2,3,4\}$. So Lemma~\ref{joinproduct} implies that  
$\cxx^{\starr(1)}\simeq CX_{1}\times\cxx^{\link(1)}\simeq\cxx^{\link(1)}$. 
Since $\link(1)=\Delta^{2}_{0}$, we can apply the previous example to obtain 
a homotopy equivalence 
\[\cxx^{K_{0}}\simeq(\Sigma X_{2}\wedge X_{3})\vee(\Sigma X_{2}\wedge X_{4}) 
      \vee(\Sigma X_{3}\wedge X_{4})\vee 2\cdot(\Sigma X_{2}\wedge X_{3}\wedge X_{4}).\] 

\noindent{\textit{Step 2}: Since $K_{1}=K_{0}\cup_{\{2,3\}}(2,3)$, 
Theorem~\ref{conefillpo} implies that there is a homotopy equivalence 
\begin{equation} 
  \label{exeq1} 
  \cxx^{K_{1}}\simeq D_{1}\vee[(X_{2}\ast X_{3}\ast X_{1})\rtimes X_{4}] 
\end{equation}  
where there is a split cofibration 
\(\nameddright{(X_{2}\ast X_{3})\rtimes X_{4}}{}{\cxx_{K_{0}}}{}{D_{1}}\). 
As 
$(X_{2}\ast X_{3})\rtimes X_{4}\simeq(\Sigma X_{2}\wedge X_{3})\vee 
      (\Sigma X_{2}\wedge X_{3}\wedge X_{4})$, 
the homotopy equivalence for $\cxx^{K_{0}}$ in Step~1 implies that 
there is a homotopy equivalence 
\[D_{1}\simeq(\Sigma X_{2}\wedge X_{4})\vee(\Sigma X_{3}\wedge X_{4}) 
     \vee(\Sigma X_{2}\wedge X_{3}\wedge X_{4}).\] 
Thus~(\ref{exeq1}) implies that there is a homotopy equivalence 
\[\cxx^{K_{1}}\simeq(\Sigma X_{2}\wedge X_{4})\vee(\Sigma X_{3}\wedge X_{4}) 
     \vee(\Sigma X_{2}\wedge X_{3}\wedge X_{4})  
     \vee(\Sigma^{2} X_{1}\wedge X_{2}\wedge X_{3}) 
     \vee(\Sigma^{2} X_{1}\wedge X_{2}\wedge X_{3}\wedge X_{4}).\] 

\noindent{\textit{Step 3}: Since $K_{2}=K_{1}\cup_{\{2,4\}}(2,4)$, 
Theorem~\ref{conefillpo} implies that there is a homotopy equivalence 
\begin{equation} 
  \label{exeq2} 
  \cxx^{K_{2}}\simeq D_{2}\vee[(X_{2}\ast X_{4}\ast X_{1})\rtimes X_{3}] 
\end{equation}  
where there is a split cofibration 
\(\nameddright{(X_{2}\ast X_{4})\rtimes X_{3}}{}{\cxx_{K_{1}}}{}{D_{2}}\). 
As 
$(X_{2}\ast X_{4})\rtimes X_{3}\simeq(\Sigma X_{2}\wedge X_{4})\vee 
      (\Sigma X_{2}\wedge X_{3}\wedge X_{4})$, 
the homotopy equivalence for $\cxx^{K_{1}}$ in Step~2 implies that 
there is a homotopy equivalence 
\[D_{2}\simeq(\Sigma X_{3}\wedge X_{4})\vee(\Sigma^{2} X_{1}\wedge X_{2}\wedge X_{3}) 
     \vee(\Sigma^{2} X_{1}\wedge X_{2}\wedge X_{3}\wedge X_{4}).\] 
Thus~(\ref{exeq2}) implies that there is a homotopy equivalence 
\[\cxx^{K}=\cxx^{K_{2}}\simeq(\Sigma X_{3}\wedge X_{4}) 
     \vee(\Sigma^{2} X_{1}\wedge X_{2}\wedge X_{3}) 
     \vee(\Sigma^{2} X_{1}\wedge X_{3}\wedge X_{4}) 
     \vee 2\cdot(\Sigma^{2} X_{1}\wedge X_{2}\wedge X_{3}\wedge X_{4}).\]

\section{Extensions of the method I: gluing along a common face} 
\label{sec:gluing}  

The basic idea behind proving Theorem~\ref{cxxshifted} was to 
present $\cxx^{K}$ as the end result of a sequence of pushouts, 
and then analyze the homotopy theory of the pushouts. In these 
terms, the key ingredient of the proof was Lemma~\ref{polemma}. 
The idea behind the method is therefore very general. One can 
look for different constructions of $K$ which translate to 
a sequence of homotopy pushouts constructing $\cxx^{K}$, whose 
homotopy theory can be analyzed. This may apply to different 
classes of complexes $K$ other than the shifted class. In this 
section we give such a construction. 

Let~$K$ be a simplicial complex on the index set $[n]$. Suppose 
$K=K_{1}\cup_{\tau} K_{2}$ for $\tau$ a simplex in~$K$. That is, $K$ 
is the result of gluing~$K_{1}$ and $K_{2}$ together along a common face. 
Relabelling the vertices if necessary, we may assume that~$K_{1}$ 
is defined on the vertices $\{1,\ldots,m\}$, $K_{2}$ is defined on 
the vertices $\{m-l+1,\ldots,n\}$ and $\tau$ is defined on the 
vertices $\{m-l+1,\ldots,m\}$. Let $\overline{K}_{1}$, $\overline{K}_{2}$ 
and~$\overline{\tau}$ be $K_{1}$, $K_{2}$ and $\tau$ regarded as 
simplicial complexes on $[n]$. So 
$K=\overline{K}_{1}\cup_{\overline{\tau}}\overline{K}_{2}$. 

Let $\sigma\in K_{1}$ and let $\overline{\sigma}$ be its image 
in $\overline{K}_{1}$. By definition of $\overline{\sigma}$, we have 
$i\notin\overline{\sigma}$ for $i\in\{m+1,\ldots,n\}$. Thus 
$\cxx^{\overline{\sigma}}=\cxx^{\sigma}\times X_{m+1}\times\cdots\times X_{n}$. 
Consequently, taking the union over all the faces in $\overline{K}_{1}$, we obtain 
\[\cxx^{\overline{K}_{1}}=\cxx^{K_{1}}\times X_{m+1}\times\cdots\times X_{n}.\] 
Similarly, we have 
\[\cxx^{\overline{K}_{2}}=X_{1}\times\cdots\times X_{m-l}\times\cxx^{K_{2}}.\] 
Since $\tau=\Delta_{m-l-1}$, we have 
$\cxx^{\tau}=CX_{m-l+1}\times\cdots\times CX_{m}$, so as above we obtain 
\[\cxx^{\overline{\tau}}=X_{1}\times\cdots\times X_{m-l}\times 
        CX_{m-l+1}\times\cdots\times CX_{m}\times X_{m+1}\times\cdots\times X_{n}.\] 
Since $K=K_{1}\cup_{\tau} K_{2}$, by Proposition~\ref{uxaunion} there 
is a pushout 
\begin{equation} 
  \label{podgrm} 
  \diagram 
     X_{1}\times\cdots\times X_{m-l}\times\cxx^{\tau}\times 
          X_{m+1}\times\cdots\times X_{n}\rto^-{a}\dto^{b}  
         & \cxx^{K_{1}}\times X_{m+1}\times\cdots\times X_{n}\dto \\ 
     X_{1}\times\cdots\times X_{m-l}\times\cxx^{K_{2}}\rto & \cxx^{K} 
  \enddiagram 
\end{equation} 
where $a$ and $b$ are coordinate-wise inclusions. 

We next identify the homotopy classes of $a$ and $b$. We use 
the Milnor-Moore convention of writing the identity map 
\(\namedright{Y}{}{Y}\) 
as $Y$. To simplify notation, let $M=X_{1}\times\cdots\times X_{m-l}$ 
and $N=X_{m+1}\times\cdots\times X_{n}$. Then the domain of $a$ 
and $b$ is $M\times\cxx^{\tau}\times N$. Since $a$ and $b$ are 
coordinate-wise inclusions, their homotopy classes are determined 
by their restrictions to $M$, $\cxx^{\tau}$ and $N$. Consider $a$. 
Since each vertex $\{i\}\in K$ for $1\leq i\leq m-l$, 
Corollary~\ref{projinclcor} implies that the restriction of $a$ to $M$ is 
null homotopic. Since $\cxx^{\tau}$ is a product of cones, it is 
contractible, so the restriction of $a$ to $\cxx^{\tau}$ is null 
homotopic. Since $a$ is a coordinate-wise inclusion, it is the identity 
map on $X_{m+1}\times\cdots\times X_{n}$. Thus 
$a\simeq\ast\times\ast\times N$. Similarly, 
$b\simeq M\times\ast\times\ast$. Thus we can rewrite~(\ref{podgrm}) 
as a pushout 
\begin{equation} 
  \label{podgrm2} 
  \diagram 
         M\times\cxx^{\tau}\times N\rto^-{f\times N}\dto^{M\times g} 
                & \cxx^{K_{1}}\times N\dto \\ 
         M\times\cxx^{K_{2}}\rto & \cxxk 
  \enddiagram 
\end{equation} 
where $f$ and $g$ are null homotopic. 

We wish to identify the homotopy type of \cxxk. To do so we use 
a general lemma, proved in~\cite{GT1}. Let $A$ and $B$ 
be spaces. Recall that the \emph{join} of $A$ and $B$ is 
$A\ast B=A\times I\times B/\sim$, where $(x,0,y_{1})\sim (x,0,y_{2})$ 
and $(x_{1},1,y)\sim (x_{2},1,y)$, and there is a homotopy equivalence 
$A\ast B\simeq\Sigma A\wedge B$. The \emph{left half-smash} 
of $A$ and $B$ is $A\ltimes B=A\times B/\sim$ where $(a,\ast)\sim\ast$, 
and the \emph{right half-smash} of $A$ and $B$ is 
$A\rtimes B=A\times B/\sim$ where $(\ast,b)\sim\ast$.  

\begin{lemma} 
   \label{polemma1} 
   Let 
   \[\diagram 
              A\times B\rto^-{\ast\times B}\dto^{A\times\ast} 
                     & C\times B\dto \\ 
              A\times D\rto & Q 
     \enddiagram\] 
   be a homotopy pushout. Then there is a homotopy equivalence 
   \[Q\simeq (A\ast B)\vee (A\ltimes D)\vee (C\rtimes B).\] 
   $\qqed$ 
\end{lemma} 

Lemma~\ref{polemma} does not quite fit the setup in~(\ref{podgrm2}). 
To get this we need a slight adjustment. 

\begin{lemma} 
   \label{polemma2} 
   Let 
   \[\diagram 
              A\times E\times B\rto^-{f\times B}\dto^{A\times g} 
                     & C\times B\dto \\ 
              A\times D\rto & Q 
     \enddiagram\] 
   be a homotopy pushout, where $E$ is contractible and $f$ and $g$ 
   are null homotopic. Then there is a homotopy equivalence 
   \[Q\simeq (A\ast B)\vee (A\ltimes D)\vee (C\rtimes B).\] 
\end{lemma} 
      
\begin{proof} 
Let 
\(j\colon\namedright{A\times B}{}{A\times E\times B}\) 
be the inclusion. Observe that $(f\times B)\circ j\simeq\ast\times B$ 
and $(A\times g)\circ j\simeq A\times\ast$. Thus as $E$ is contractible, 
$j$ is a homotopy equivalence and the pushout in the statement of this  
lemma is equivalent, up to homotopy, to the pushout in the 
statement of Lemma~\ref{polemma}. The homotopy equivalence 
for $Q$ now follows. 
\end{proof}        

Applying Lemma~\ref{polemma2} to the pushout in~(\ref{podgrm2}), 
we obtain the following. 

\begin{theorem} 
   \label{commonface} 
   Let $K$ be a simplicial complex on the index set $[n]$. Suppose that  
   $K=K_{1}\cup_{\tau} K_{2}$ where $\tau$ is a common face 
   of $K_{1}$ and $K_{2}$. Then there is a homotopy equivalence 
   \[\cxxk\simeq (M\ast N)\vee (\cxx^{K_{1}}\rtimes N)\vee 
           (M\ltimes\cxx^{K_{2}})\vee\] 
   where $M=X_{1}\times\cdots\times X_{m-l}$ and 
   $N=X_{m+1}\times\cdots\times X_{n}$.~$\qqed$ 
\end{theorem}  

For example, let $K$ be the simplicial complex in Example~\ref{shiftedexamples}~(1). 
Then $K$ can be obtained by gluing two copies of $\Delta^{2}_{1}$  
along an edge. Specifically, $K=K_{1}\cup_{\tau} K_{2}$ where 
$K_{1}$ is the simplicial complex on vertices $\{1,2,3\}$ having 
edges $\{(1,2), (1,3), (2,3)\}$; $K_{2}$ is the simplicial complex on 
vertices $\{1,2,4\}$ having edges $\{(1,2), (1,4), (2,4)\}$; and 
$\tau$ is the edge $(1,2)$. Since $K_{1}=\Delta^{2}_{1}$, 
Proposition~\ref{join} implies that 
$\cxx^{K_{1}}\simeq\Sigma^{2} X_{1}\wedge X_{2}\wedge X_{3}$. 
Similarly, $\cxx{K}_{2}\simeq\Sigma^{2} X_{1}\wedge X_{2}\wedge X_{4}$. 
In general, the space~$M$ is the product of the $X_{i}$'s where 
$i$ is not a vertex of $K_{2}$, and similarly for $N$ and $K_{1}$. 
So in this case $M=X_{3}$ and $N=X_{4}$. Theorem~\ref{commonface} 
therefore implies that there is a homotopy equivalence 
\[\cxx^{K}\simeq (X_{3}\ast X_{4})\vee 
       ((\Sigma^{2} X_{1}\wedge X_{2}\wedge X_{3})\rtimes X_{4})\vee 
       (X_{3}\ltimes\Sigma^{2} X_{1}\wedge X_{2}\wedge X_{4}).\] 
In general, there is a homotopy equivalence 
$\Sigma A\ltimes B\simeq\Sigma A\vee(\Sigma A\wedge B)$, and 
similarly for the right half-smash. Thus in our case we obtain a 
homotopy equivalence 
\[\cxx^{K}\simeq (\Sigma X_{3}\wedge X_{4})\vee 
      (\Sigma^{2} X_{1}\wedge X_{2}\wedge X_{3})\vee 
      (\Sigma^{2} X_{1}\wedge X_{2}\wedge X_{4})\vee 
      2\cdot(\Sigma^{2} X_{1}\wedge X_{2}\wedge X_{3}\wedge X_{4}).\] 
This matches the answer in Section~\ref{sec:examples}.  

Note that in this case $K$ is shifted, but Theorem~\ref{commonface} 
also applies to nonshifted complexes. For example, let $L_{2}$ be 
the previous example of two copies of $\Delta^{2}_{1}$ glued along 
an edge. Now glue another copy of $\Delta^{2}_{1}$ to $L_{2}$ along 
an edge. Then we obtain a complex $L_{3}$ on $[5]$ which is not shifted, 
but the homotopy equivalence for $\cxx^{L_{2}}$ and Theorem~\ref{commonface} 
imply that $\cxx^{L_{3}}\in\W_{5}$. In the same way, we could 
continue to iteratively glue in more copies of $\Delta^{2}_{1}$ along 
a common edge and obtain non-shifted complexes $L_{n-2}$ on $[n]$ 
with $\cxx^{L_{n-2}}\in\W_{n}$.

\section{Extensions of the method II: the simplicial wedge construction} 
\label{sec:wedge} 

Let $K$ be a simplicial complex on vertices $\{v_{1},\ldots,v_{n}\}$. 
Fix a vertex $v_{i}$. Define a new simplicial complex $K(v_{i})$ on 
the $n+1$ vertices $\{v_{1},\ldots,v_{i-1},v_{i,1},v_{i,2},v_{i+1},\ldots,v_{n}\}$ by  
\[K(v_{i})=\{v_{i,1},v_{i,2}\}\ast\link_{K}(v_{i})\cup 
       (v_{i,1},v_{i,2})\ast\rest_{K}(v_{i}).\] 
The simplicial complex $K(v_{i})$ is called the \emph{simplicial wedge} 
of $K$ on $v_{i}$. This construction arises in combinatorics (see~\cite{PB}) 
and has the important property that if $K$ is the boundary of the dual of a 
polytope then so is $K(v_{i})$. 

As in~\cite{BBCG2}, the construction can be iterated. To set this up, 
let $(1,\ldots,1)$ be a sequence of $n$ copies of~$1$, corresponding 
to the vertex set $\{v_{1},\ldots,v_{n}\}$. The vertex doubling operation 
of~$v_{i}$ in the simplicial wedge construction gives a new vertex set 
for $K(v_{i})$ -- listed above -- to which we associate the sequence 
$(1,\ldots,1,2,1,\ldots,1)$ of length $n$, where the $2$ appears in 
position $i$. The sequence $(1,\ldots,1,3,1,\ldots,1)$ then corresponds 
to either the simplicial wedge $(K(v_{i}))(v_{i,1})$ or to $(K(v_{i}))(v_{i,2})$. However, 
these two complexes are equivalent, so the choice of vertex $v_{i,1},v_{i,2}$ 
does not matter. More generally, let $J=(j_{1},\ldots,j_{n})$ be a sequence 
of positive integers. Define a new simplicial complex~$K(J)$ on vertices 
\[\{v_{1,1},\ldots,v_{1,j_{1}},v_{2,1},\ldots,v_{2,j_{2}},\ldots 
           v_{n,1},\ldots,v_{n,j_{n}}\}\] 
by iteratively applying the simplicial wedge construction, starting 
with $K$. 

We will show that if $K$ is shifted then, for any $J$, there is a 
homotopy decomposition 
$\cxx^{K(J)}\simeq\left(\bigvee_{I\notin K(J)} 
       \vert K(J)_{I}\vert\ast\widehat{X}^{I}\right)$. 
This improves on Theorem~\ref{cxxshifted} because the class 
of simplicial complexes obtained from shifted complexes by 
simplicial wedge constructions is strictly larger than the class 
of shifted complexes. We give an example to illustrate this. 

\begin{example} 
Let $K$ be the simplicial complex consisting of vertices $\{1,2,3,4\}$  
and edges $\{(1,2),(1,3)\}$. Observe that $K$ is shifted. Apply the simplicial 
wedge product which doubles vertex $4$, that is, let $J=(1,1,1,2)$. 
Then $K(J)$ is a simplicial complex on vertices $\{1,2,3,4a,4b\}$. 
We have $K(J)=(4a,4b)\ast\link_{K}(4)\cup\{4a,4b\}\ast\rest_{K}(4)$. 
Here, $\link_{K}(4)=\emptyset$ so $(4a,4b)\ast\link_{K}(4)=(4a,4b)$.  
As well, $\rest_{K}(4)$ has vertices $\{1,2,3\}$ and edges $\{(1,2),(1,3)\}$, 
so $\{4a,4b\}\ast\rest_{K}(4)$ has vertices $\{1,2,3,4a,4b\}$ and is the 
union of the faces $\{(1,2,4a), (1,2,4b), (1,3,4a), (1,3,4b)\}$. 

We claim that $K(J)$ is not shifted. Observe that the edge $(2,3)\notin K(J)$, 
but every other possible edge is in $K(J)$. That is, $(x,y)\in K(J)$ for 
every $x,y\in\{1,2,3,4a,4b\}$ except $(2,3)$. Thus with the ordering 
$1<2<3<4a<4b$, $K(J)$ does not satisfy the shifted condition as 
$(2,4a)\in K(J)$ would imply that $(2,3)\in K(J)$. So if $K(J)$ is to be 
shifted, we must reorder the vertices. Let  
$\{1^{\prime}, 2^{\prime}, 3^{\prime}, 4^{\prime}, 5^{\prime}\}$ 
be the new labels of the vertices. To satisfy the shifted condition 
we need to send the vertices $\{2,3\}$ to $\{4^{\prime},5^{\prime}\}$. 
The vertices $\{1,4a,4b\}$ are therefore sent to 
$\{1^{\prime}, 2^{\prime}, 3^{\prime}\}$. Now observe that the 
face $(1,4a,4b)\notin K(J)$. Thus in the new ordering, the face 
$(1^{\prime},2^{\prime},3^{\prime})\notin K(J)$. The shifted 
condition therefore implies that no $2$-dimensional faces are 
in $K(J)$, a contradiction. Hence there is no reordering of the 
vertices of $K(J)$ for which the shifted condition holds. Hence 
$K(J)$ is not shifted. 
\end{example} 

\begin{proposition} 
   \label{cxxwedge} 
   Let $K$ be a shifted complex on $n$ vertices. If $v_{i}\in K$ is a vertex, 
   then there is a homotopy equivalence 
   \[\cxx^{K(v_{i})}\simeq\left(\bigvee_{I\notin K(v_{i})} 
       \vert K(v_{i})_{I}\vert\ast\widehat{X}^{I}\right).\] 
\end{proposition} 

\begin{proof} 
We have 
$K(v_{i})=(v_{i,1},v_{i,2})\ast\link_{K}(v_{i})\cup\{v_{i,1},v_{i,2}\}\ast\rest_{K}(v_{i})$. 
Recall that, by definition, $\link_{K}(v_{i})\subseteq\rest_{K}(v_{i})$. Thus 
$(v_{i,1},v_{i,2})\ast\link_{K}(v_{i})\cap\{v_{i,1},v_{i,2}\}\ast\rest_{K}(v_{i})= 
     \{v_{i,1},v_{i,2}\}\ast\link_{K}(v_{i})$. 
We can therefore regard $K(v_{i})$ as a pushout 
\[\diagram 
       \{v_{i,1},v_{i,2}\}\ast\link_{K}(v_{i})\rto\dto 
             & \{v_{i,1},v_{i,2}\}\ast\rest_{K}(v_{i})\dto \\ 
       (v_{i,1},v_{i,2})\ast\link_{K}(v_{i})\rto & K(v_{i}). 
  \enddiagram\] 
Thus Proposition~\ref{uxaunion} implies that there is a pushout of spaces 
\begin{equation} 
  \label{wedgedgrm1} 
  \diagram 
        \cxx^{\{v_{i,1},v_{i,2}\}}\times\cxx^{\link_{K}(v_{i})} 
                \rto^-{1\times g}\dto^{f\times 1} 
             & \cxx^{\{v_{i,1},v_{i,2}\}}\times\cxx^{\rest_{K}(v_{i})}\dto \\ 
       \cxx^{(v_{i,1},v_{i,2})}\times\cxx^{\link_{K}(v_{i})}\rto & \cxx^{K(v_{i})} 
  \enddiagram 
\end{equation} 
where $f$ and $g$ are inclusions. Since $\{v_{i,1},v_{i,2}\}$ is a copy 
of $\Delta^{1}_{0}$, Proposition~\ref{join} implies that 
$\cxx^{\{v_{i,1},v_{i,2}\}}\simeq\Sigma X_{v_{i,1}}\ast X_{v_{i,2}}$. 
Since $\cxx^{(v_{i,1},v_{i,2})}$ is a copy of $\Delta^{1}$, the definition 
of the polyhedral product implies that 
$\cxx^{(v_{i,1},v_{i,2})}\simeq CX_{v_{i,1}}\times CX_{v_{i,2}}\simeq\ast$. 
Thus, up to homotopy equivalences, (\ref{wedgedgrm1})~is equivalent 
to the homotopy pushout 
\begin{equation} 
  \label{wedgedgrm2} 
  \diagram 
        (\Sigma X_{v_{i,1}}\ast X_{v_{i,2}})\times\cxx^{\link_{K}(v_{i})} 
                \rto^-{1\times g}\dto^{\pi_{2}} 
             & (\Sigma X_{v_{i,1}}\ast X_{v_{i,2}})\times\cxx^{\rest_{K}(v_{i})}\dto \\ 
       \cxx^{\link_{K}(v_{i})}\rto & \cxx^{K(v_{i})}.  
  \enddiagram 
\end{equation} 

We will compare~(\ref{wedgedgrm2}) to another pushout. Since $K$ 
is shifted, there is a pushout 
\[\diagram 
        \link_{K}(v_{i})\rto\dto & \rest_{K}(v_{i})\dto \\ 
        \starr_{K}(v_{i})\rto & K. 
  \enddiagram\] 
Recall that $\starr_{K}(v_{i})=(v_{i})\ast\link_{K}(v_{i})$. Thus 
Proposition~\ref{uxaunion} implies that there is a pushout of spaces 
\[\diagram 
      X_{v_{i}}\times\cxx^{\link_{K}(v_{i})}\rto^-{1\times g}\dto^{f\times 1} 
            & X_{v_{i}}\times\cxx^{\rest_{K}(v_{i})}\dto \\ 
     CX_{v_{i}}\times\cxx^{\link_{K}(v_{i})}\rto & \cxx^{K} 
  \enddiagram\] 
where $f$ and $g$ are inclusions. Since $CX_{v_{i}}$ is contractible, 
we obtain a homotopy pushout diagram 
\begin{equation} 
  \label{wedgedgrm3} 
  \diagram 
      X_{v_{i}}\times\cxx^{\link_{K}(v_{i})}\rto^-{1\times g}\dto^{\pi_{2}} 
            & X_{v_{i}}\times\cxx^{\rest_{K}(v_{i})}\dto \\ 
     \cxx^{\link_{K}(v_{i})}\rto & \cxx^{K}. 
  \enddiagram 
\end{equation} 
Observe that~(\ref{wedgedgrm2}) and~(\ref{wedgedgrm3}) have exactly 
the same format. That is, (\ref{wedgedgrm2}) is precisely~(\ref{wedgedgrm3}) 
with the space $X_{v_{i}}$ replaced by $\Sigma X_{v_{i,1}}\ast X_{v_{i,2}}$. 
Thus exactly the same procedure used in the proof of Theorem~\ref{cxxshifted} 
can be applied to show that 
$\cxx^{K(J)}\simeq\left(\bigvee_{I\notin K(J)} 
       \vert K(J)_{I}\vert\ast\widehat{X}^{I}\right)$. 
\end{proof} 

Since $K(v_{i})$ need not be shifted, Proposition~\ref{cxxwedge} cannot 
be used to produce iterative decompositions of $\cxx^{K(J)}$ for any $J$. 
However, the methods involved should be adaptable to the more general 
case. So we conjecture that for any $J$, there is a homotopy equivalence 
\[\cxx^{K(J)}\simeq\left(\bigvee_{I\notin K(J)} 
       \vert K(J)_{I}\vert\ast\widehat{X}^{I}\right).\]

%%% The bibliography %%%
\bibliographystyle{amsalpha}

\end{document}